\documentclass{article}

\usepackage[a4paper, margin=1in]{geometry}

\usepackage[english]{babel}

\usepackage{multicol}

\usepackage{setspace}

\usepackage{booktabs}

\usepackage{amssymb}

\usepackage{adjustbox}

\usepackage{graphicx}

\usepackage{amsmath}

\usepackage{mwe}

\usepackage{colortbl}

\usepackage{caption}

\usepackage{mathrsfs}

\usepackage{bbding}

\usepackage{float}

\usepackage{multirow}

\usepackage{stix}

\usepackage{tikz-cd}

\usetikzlibrary {chains}
\usetikzlibrary{arrows.meta}

\usepackage{amsthm}

\newtheorem{theorem}{Theorem}

\newtheorem{definition}{Definition}
\newtheorem{prop}{Proposition}

\newtheorem{lemma}{Lemma}

\newtheorem{cor}{Corollary}

\newtheorem{example}{Example}

\newcommand{\lfis}{{\bf LFI}s}
\newcommand{\mT}{\mathcal{FT}}

\usepackage{authblk}

\title{\textbf{Absolutely Free Hyperalgebras}}
\author{Coniglio, Marcelo E.\thanks{coniglio@unicamp.br} }

\author{Toledo, Guilherme V.\thanks{guivtoledo@gmail.com}}

\affil{Institute of Philosophy and the Humanities - IFCH and\\
Centre for Logic, Epistemology and The History of Science - CLE\\
University of Campinas - Unicamp\\
Campinas, SP, Brazil}

\usepackage{fancyhdr}
\providecommand{\keywords}[1]{\textbf{\textit{Keywords:}} #1}
\begin{document}

\setcounter{page}{1}     

\maketitle

\begin{abstract}
In abstract algebraic logic, many systems, such as those paraconsistent logics taking inspiration from da Costa's hierarchy, are not algebraizable by even the broadest standard methodologies, as that of Blok and Pigozzi. However, these  logics can be semantically characterized by means of non-deterministic algebraic structures such as Nmatrices, RNmatrices and swap structures. These structures are based on multialgebras, which  generalize algebras by allowing the result of an operation to assume a non-empty set of values. This leads to an interest in exploring the foundations of multialgebras applied to the study of logic systems.

It is well known from universal algebra that, for every signature $\Sigma$, there exist algebras over $\Sigma$ which are absolutely free, meaning that they do not satisfy any identities or, alternatively, satisfy the universal mapping property for the class of $\Sigma$-algebras. Furthermore, once we fix a cardinality of the generating set, they are, up to isomorphisms, unique, and equal to algebras of terms (or propositional formulas, in the context of logic). Equivalently, the forgetful functor, from the category of $\Sigma$-algebras to $\textbf{Set}$, has a left adjoint. This result does not extend to multialgebras. Not only multialgebras satisfying the universal mapping property do not exist, but the forgetful functor $\mathcal{U}$, from the category of $\Sigma$-multialgebras to $\textbf{Set}$, does not have a left adjoint.

In this paper we generalize, in a natural way, algebras of terms to multialgebras of terms, whose family of submultialgebras enjoys many properties of the former. One example is that, to every pair consisting of a function, from a multialgebra of terms to another multialgebra, and a collection of choices (which selects how a homomorphism approaches indeterminacies), it corresponds a unique homomorphism, which ressembles the universal mapping property. Another example is that the multialgebras of terms are generated by a set that may be viewed as a strong basis, which we call the ground of the multialgebra. Submultialgebras of multialgebras of terms are what we call weakly free multialgebras. Finally, with these definitions at hand, we offer a simple proof that multialgebras with the universal mapping property for the class of all multialgebras do not exist and that $\mathcal{U}$ does not have a left adjoint.
\end{abstract}


\section{Introduction}
An interesting and fruitful strategy in contemporary formal logic is trying to find  an algebraic counterpart for a given logic or family of logics. This is the main goal of the area of mathematical logic known as {\em algebraic logic}, or {\em abstract algebraic logic} (AAL) in a more general perspective. 

The idea behind (traditional) algebraic logic is to develop an algebraic study of a given class of models (algebras) associated to a given logic. For instance, it can be insightful to study the relationship between Boolean (Heyting, respectively) algebras and propositional classical (intuitionistic, respectively)  logic, while an important area of mathematical fuzzy logic deals with the relationship between fuzzy logics and certain classes of residuated lattices. In turn,  AAL is more interested in analyzing and classifying the algebraization methods {\em per se}. As one would expect, the scope of (abstract) algebraic logic is far from being universal: there are important classes of logics which lie outside the usual methods and techniques of AAL. 

A good source of examples to this phenomenon can be found in the field of paraconsistency. Because of this, certain classes of paraconsistent logics, as the ones known as {\em logics of formal inconsistency},\footnote{\lfis, introduced in~\cite{CM} and coming from the tradition of da Costa's approach to paraconsistency (\cite{dC63})} are characterized by means of semantics of non-deterministic character such as non-deterministic matrices, Fidel structures or swap structures   (see for instance~\cite{CC16}). Besides giving a semantical characterization, as well as a decision procedure, for these logics, such non-deterministic structures constitute an interesting object of study by themselves  (see, for instance, \cite[Chapter~6]{CC16}, \cite{CFG} and~\cite{CT20}). 

It is worth observing that non-deterministic matrix semantics (introduced in~\cite{AvronLev}) and, more generally, swap structures semantics, are (classes of)  {\em multialgebras} equipped with a subset of designated elements of their domains, which generalize the very idea of logical matrices.  Multialgebras, also known as {\em hyperalgebras} or {\em non-deterministic algebras}, introduced in \cite{Marty}, generalize the concept of algebra by replacing operations by {\em multioperations} (or {\em hyperoperations}), whose results assume multiple values, that is, a subset of the universe. Here, we will restrict ourselves to multialgebras whose operations cannot return an empty set of values, which is a common requirement when working with non-classical logics and their semantics.

%  Thus, the study of this class of structures involves the formal analysis of multialgebras, which is the subject of the present paper.

In the realm of universal algebra, it is a well known result (\cite{Burris}) that there exist algebras $\mathcal{A}$ over a given signature $\Sigma$ that satisfy the so-called {\em universal mapping property} for the class of all $\Sigma$-algebras over some subset $X$ of their universe $A$. This property says that, for any other $\Sigma$-algebra $\mathcal{B}$ with universe $B$ and  any function $f:X\rightarrow B$, there exists exactly one homomorphism $\overline{f}$ between $\mathcal{A}$ and $\mathcal{B}$ that extends $f$.  Such algebras are called (absolutely) free $\Sigma$-algebras generated by $X$. Moreover, any free $\Sigma$-algebra generated by $X$ is isomorphic to the $\Sigma$-algebra of terms over $X$, which will be denoted here  by $\textbf{T}(\Sigma, X)$. Thus, free algebras are unique up to isomorphisms. In the language of categories, the existence of free $\Sigma$-algebras  means that the forgetful functor $U:\textbf{Alg}(\Sigma)\rightarrow\textbf{Set}$, from the category of $\Sigma$-algebras to the category of sets, has a left adjoint $F$, associating to a set $X$ any $\Sigma$-algebra with the universal mapping property over $X$ (which, as mentioned above, can be taken as being  $\textbf{T}(\Sigma, X)$). 

While algebras satisfying the universal mapping property always exist, and are (up to isomorphisms) algebras of terms, the situation is quite different in the context of multialgebras. 
%We make a point of separating the concepts of an algebra of terms and an algebra satisfying the universal mapping property, although they are equivalent on the classical context of universal algebra, since they do not extend well to the context of multialgebras. 
Indeed, it is well-known that multialgebras satisfying the universal mapping property do not exist, and so the forgetful functor $\mathcal{U}:\textbf{MAlg}(\Sigma)\rightarrow\textbf{Set}$, from the category of multialgebras over the signature $\Sigma$ to the category of sets, does not have a left adjoint. This means that any possible  ``multialgebra of terms'' generalizing in some sense the notion of algebra of terms to the category of multialgebras necessarily will not satisfy the  universal mapping property. A new prof of this fact will be given in Section~\ref{UMP}.

% This means that that there is no object such as a ``multialgebra of terms'' satisfying the  universal mapping property.

The aim of this paper is proposing a very natural generalization to the category of multialgebras of the concept of algebra of terms by means of a family $\mT(\Sigma,\mathcal{V})$ of multialgebras of terms indexed by  the cardinals $\kappa>0$. Any submultialgebra of a multiagebra of this family  satisfies several equivalent characterizations, which are  necessarily weaker than the standard characterization of absolutely free algebras by means of the universal mapping property. 
We propose the novel notion of {\em weakly free $\Sigma$-multialgebras} as those multialgebras satifying any, and therefore all, of these weaker conditions. In particular, all of them are isomorphic to a submultialgebra of a multialgebra in  $\mT(\Sigma,\mathcal{V})$, for some $\mathcal{V}$.

This paper is organized as follows:  Section~\ref{Definition} proposes a natural notion of multialgebras of (non-deterministic) terms. In Section~\ref{Characterizations}, five equivalent characterizations of the submultialgebras of multialgebras of terms are given, which lead to the notion of weakly free multialgebras. In Section~\ref{UMP} we apply one of the five characterizations obtained in Section~\ref{Characterizations} to offer a simple proof of the well-known result which states that the category $\textbf{MAlg}(\Sigma)$ of multialgebras does not have free objects. Finally, some conclusions are provided in Section~\ref{conclusions}.

\section{Multialgebras of non-deterministic terms} \label{Definition}

This section introduces the first main notion proposed in the paper:  multialgebras of (non-deterministic) terms. As we shall see, a generalization to the category of multialgebras of the concept of algebra of terms is attained by means of a family of multialgebras of terms indexed by all the cardinals $\kappa > 0$, instead of considering a single object. This reveals the complexity  required by adapting the notion of free objects to the category   of multialgebras:  all the possible sizes for the outputs of the multioperators, assuming that the outputs consist of sets of terms instead of terms,  should be considered. In this sense, $\kappa$ represents the maximum of such sizes in a given multialgebra of terms.
Before introducing the definition itself, some standard notions will be recalled.

A {\em signature} is a collection $\Sigma=\{\Sigma_{n}\}_{n\in \mathbb{N}}$ of (possibly empty) sets $\Sigma_{n}$. Elements of $\Sigma_n$ are {\em functional symbols} of arity $n$. We will denote by $\Sigma$ either the collection itself or, when there is no risk of confusion, the set $\bigcup_{n\in\mathbb{N}}\Sigma_{n}$.

A {\em $\Sigma$-multialgebra}, or {\em multialgebra}, is a pair $\mathcal{A}=(A, \{\sigma_{\mathcal{A}}\}_{\sigma\in\Sigma})$, where $A$ is a non-empty set (the {\em universe} of $\mathcal{A}$) and $\{\sigma_{\mathcal{A}}\}_{\sigma\in\Sigma}$ is a collection of functions indexed by $\bigcup_{n\in\mathbb{N}}\Sigma_{n}$ such that, if $\sigma\in\Sigma_{n}$, $\sigma_{\mathcal{A}}$ is a function of the form
\begin{equation*}\sigma_{\mathcal{A}}:A^{n}\rightarrow\mathcal{P}(A)\setminus\{\emptyset\},\end{equation*}
that is, an $n-$ary function from $A$ to the set of non-empty subsets of $A$.

A {\em homomorphism} between two $\Sigma$-multialgebras $\mathcal{A}=(A, \{\sigma_{\mathcal{A}}\}_{\sigma\in\Sigma})$ and $\mathcal{B}=(B, \{\sigma_{\mathcal{B}}\}_{\sigma\in\Sigma})$ is a function $f:A\rightarrow B$ such that, for all $n\in\mathbb{N}$, $\sigma\in\Sigma_{n}$ and elements $a_{1}, \ldots  , a_{n}\in A$,
\begin{equation*}\{f(a)\ : \  a\in \sigma_{\mathcal{A}}(a_{1}, \ldots  , a_{n})\}\subseteq \sigma_{\mathcal{B}}(f(a_{1}), \ldots  , f(a_{n})).\end{equation*}
When  in the previous relation we replace inclusion by equality, we say that $f$ is a {\em full homomorphism}. To denote that the function $f$ is a homomorphism from $\mathcal{A}$ to $\mathcal{B}$, we write $f:\mathcal{A}\rightarrow \mathcal{B}$. If the homomorphism $f:\mathcal{A}\rightarrow \mathcal{B}$ is injective, we call it a {\em monomorphism} and, if it is surjective, we call it an {\em epimorphism}. A bijective full homomorphism will be called an {\em isomorphism}.  

The class of all $\Sigma$-multialgebras, equipped with the homomorphisms between them (where composition and identity homomorphism are as in the category of sets), becomes the category $\textbf{MAlg}(\Sigma)$. In this category, the epics are precisely the epimorphisms, while any monomorphism is a monic. In turn, isomorphisms, as defined above, are exactly the isomorphisms in the categorical sense (see, for instance, \cite{CFG}, Section~2). Notice, however, that is not known whether all monics are monomorphisms. Any standard $\Sigma$-algebra can be seen as a $\Sigma$-multialgebra in which the operators return singletons. It is easy to see that the category of $\Sigma$-algebras is a full subcategory of $\textbf{MAlg}(\Sigma)$.

Given two $\Sigma$-multialgebras $\mathcal{A}=(A, \{\sigma_{\mathcal{A}}\}_{\sigma\in\Sigma})$ and $\mathcal{B}=(B, \{\sigma_{\mathcal{B}}\}_{\sigma\in\Sigma})$ such that $B\subseteq A$, we say $\mathcal{B}$ is a {\em submultialgebra} of $\mathcal{A}$ if the identity function $id:B\rightarrow A$ is a homomorphism from $\mathcal{B}$ to $\mathcal{A}$ (being therefore a monic). That is, for every $b_{1}, \ldots  , b_{n} \in B$,
\begin{equation*}\sigma_{\mathcal{B}}(b_{1}, \ldots  , b_{n})\subseteq \sigma_{\mathcal{A}}(b_{1}, \ldots  , b_{n}).\end{equation*}

Given a set $\mathcal{V}$ of  variables and a signature $\Sigma=\{\Sigma_{n}\}_{n\in\mathbb{N}}$, the algebra of terms generated by $\mathcal{V}$ over $\Sigma$ will be denoted by $\textbf{T}(\Sigma, \mathcal{V})$, and its universe will be denoted by $T(\Sigma, \mathcal{V})$. The set $T(\Sigma, \mathcal{V})$ is the smallest subset $X$ of the set of finite, non-empty sequences over $\mathcal{V} \cup \bigcup_{n\in\mathbb{N}}\Sigma_{n}$ such that:

\begin{enumerate}
\item $\mathcal{V} \cup \Sigma_0 \subseteq X$;

\item $\sigma\alpha_{1}\ldots \alpha_{n} \in X$, whenever $n \geq 1$, $\sigma\in\Sigma_{n}$ and $\alpha_{1}, \ldots  , \alpha_{n}$ in $X$.
\end{enumerate}

The set $T(\Sigma, \mathcal{V})$ becomes the $\Sigma$-algebra $\textbf{T}(\Sigma, \mathcal{V})$ when we define, for any $\sigma\in\Sigma_{n}$ and terms $\alpha_{1}, \ldots  , \alpha_{n}$ in $T(\Sigma, \mathcal{V})$, 
\begin{equation*}\sigma_{\textbf{T}(\Sigma, \mathcal{V})}(\alpha_{1}, \ldots  , \alpha_{n})=\sigma\alpha_{1}\ldots \alpha_{n}.\end{equation*}

We define the order (or complexity) $\sf o(\alpha)$ of a term $\alpha$ of $\textbf{T}(\Sigma, \mathcal{V})$ as: $\sf o(\alpha)=0$ if $\alpha\in \mathcal{V} \cup \Sigma_{0}$; and $\sf o(\sigma\alpha_{1} \ldots  \alpha_{n})=1+\max\{\sf o(\alpha_{1}), \ldots  , \sf o(\alpha_{n})\}$.

\begin{definition}
Given a signature $\Sigma$ and a cardinal $\kappa>0$, the {\it expanded signature} $\Sigma^{\kappa}=\{\Sigma_{n}^{\kappa}\}_{n\in\mathbb{N}}$ is the signature such that $\Sigma_{n}^{\kappa}=\Sigma_{n}\times \kappa$, where we will denote the pair $(\sigma, \beta)$ by $\sigma^{\beta}$ for $\sigma\in\Sigma$ and $\beta\in\kappa$. 
\end{definition}

We demand that $\kappa$ is greater than zero, which guarantees that, if $\Sigma$ is non-empty, so is $\Sigma^{\kappa}$.

\begin{definition} \label{free-mterms}
Given a set of  variables $\mathcal{V}$, a signature $\Sigma$ and a cardinal $\kappa>0$, we define the {\em $\kappa$-branching $\Sigma$-multialgebra of non-deterministic terms}, or simply $\kappa$-branching multialgebra of terms, when $\Sigma$ is obvious from the context, as
%the $\Sigma$-{\it multialgebra of non-deterministic terms of size $\kappa$}, or simply multialgebra of terms of size $\kappa$, as
\begin{equation*}\textbf{mT}(\Sigma, \mathcal{V}, \kappa)=(T(\Sigma^{\kappa}, \mathcal{V}), \{\sigma_{\textbf{mT}(\Sigma, \mathcal{V}, \kappa)}\}_{\sigma\in\Sigma}),\end{equation*}
with universe $T(\Sigma^{\kappa}, \mathcal{V})$ and such that, for $\sigma\in \Sigma_{n}$ and $\alpha_{1}, \ldots  , \alpha_{n}\in T(\Sigma^{\kappa}, \mathcal{V})$,
\begin{equation*}\sigma_{\textbf{mT}(\Sigma, \mathcal{V}, \kappa)}(\alpha_{1}, \ldots  , \alpha_{n})=\{\sigma^{\beta}\alpha_{1} \ldots   \alpha_{n} \ : \  \beta\in\kappa\}.\end{equation*}
Let $\mT(\Sigma, \mathcal{V})=(\textbf{mT}(\Sigma, \mathcal{V}, \kappa))_{\kappa\geq 1}$ be the family of such multialgebras of terms.
\end{definition}

The intuition behind this definition is that connecting given terms $\alpha_{1}, \ldots,\alpha_{n}$ with a functional symbol $\sigma$ can, in a broader interpretation taking into account non-determinism, return many terms with the same general shape, namely $\sigma\alpha_{1} \ldots \alpha_{n}$. All of such terms are constructed with  functional symbols $\sigma^{\beta}$, and the collection of  them (for $\beta \in \kappa$) corresponds to the non-deterministic term generated from the given input.

In the general case, not all functional symbols should return the same number  $\kappa$ of generalized terms. Because of this, the submultialgebras of $\textbf{mT}(\Sigma, \mathcal{V}, \kappa)$ will be considered, where the cardinality of the outputs will vary as long as it is bounded by $\kappa$. %Thus, the collection formed by the $\Sigma$-algebra of terms together with its subalgebras is generalized to multialgebras through the class of submultialgebras of the family of multialgebras $\textbf{mT}(\Sigma, \mathcal{V}, \kappa)$, where $\kappa> 0$.
Here, we will restrict ourselves to the cases where $\Sigma_{0}\neq\emptyset$ or $\mathcal{V}\neq\emptyset$, so that $\textbf{mT}(\Sigma, \mathcal{V}, \kappa)$ is always well defined. 

The order of an element $\alpha$ of $\textbf{mT}(\Sigma, \mathcal{V}, \kappa)$ is, by definition, its order as an element of $T(\Sigma^{\kappa}, \mathcal{V})$. Notice that, if 
\begin{equation*}\sigma_{\textbf{mT}(\Sigma, \mathcal{V}, \kappa)}(\alpha_{1}, \ldots  , \alpha_{n})\cap \theta_{\textbf{mT}(\Sigma, \mathcal{V}, \kappa)}(\beta_{1}, \ldots   ,\beta_{m})\neq\emptyset,\end{equation*}
then $\sigma=\theta$, $n=m$ and $\alpha_{1}=\beta_{1}$, \ldots   , $\alpha_{n}=\beta_{m}$, since if the intersection is not empty there are $\beta, \gamma\in\kappa$ such that $\sigma^{\beta}\alpha_{1} \ldots  \alpha_{n}= \theta^{\gamma}\beta_{1} \ldots \beta_{m}$ and by the structure of $T(\Sigma^{\kappa}, \mathcal{V})$ we find that $\sigma^{\beta}=\theta^{\gamma}$. 

\begin{example}
The $\Sigma$-algebras of terms $\textbf{T}(\Sigma, \mathcal{V})$, when considered as multialgebras such that $\sigma_{\textbf{T}(\Sigma, \mathcal{V})}(\alpha_{1}, \ldots  , \alpha_{n})=\{\sigma\alpha_{1} \ldots \alpha_{n}\}$, are multialgebras of terms, with $\kappa=1$. That is, $\textbf{T}(\Sigma, \mathcal{V})$ and $\textbf{mT}(\Sigma, \mathcal{V}, 1)$ are isomorphic.
\end{example}

From now on, the cardinal of a set $X$ will be denoted by $|X|$.

\begin{example}\label{s}
A {\em directed graph} is a pair $(V, A)$, with $V$ a non-empty set of elements called {\em vertices} and $A\subseteq V^{2}$ a set of elements called {\em arrows}. We say that there is an arrow from $u$ to $v$, both in $V$, if $(u,v)\in A$. We say that the $n$-tuple $(v_{1}, \ldots  , v_{n})$ is a {\em path} between $u$ and $v$ if $u=v_{1}$, $v=v_{n}$ and $(v_{i}, v_{i+1})\in A$ for every $i\in\{1, \ldots  , n-1\}$. We say that $u\in V$ has a {\em successor} if there exists $v\in V$ such that $(u,v)\in A$, and $u$ has a {\em predecessor} if there exists $v\in V$ such that $(v,u)\in A$.

A directed graph $F=(V, A)$ is a {\em forest} if, for any two $u, v\in V$, there exists at most one path between $u$ and $v$, and a forest is said to have {\em height} $\omega$ if every vertex has a successor. We state that forests of height $\omega$ are in bijection with the submultialgebras of the multialgebras of terms over the signature $\Sigma_{s}$ with exactly one operator $s$ of arity $1$.
%, such that $(\Sigma_s)_{1}=\{s\}$ and $(\Sigma_s)_{n}=\emptyset$ for $n\in\mathbb{N}\setminus\{1\}$.

Indeed, take as $\mathcal{V}$ the set of elements of $F$ that have no predecessor and define, for $u\in V$,
\begin{equation*}s_{\mathcal{A}}(u)=\{v\in V \ : \ (u,v)\in A\}.\end{equation*}
It is easy to see that the $\Sigma_{s}$-multialgebra $\mathcal{A}=(V, \{s_{\mathcal{A}}\})$, submultialgebra of $\textbf{mT}(\Sigma_{s}, \mathcal{V}, |V|)$, carries the same information that $F$.
\end{example}

\begin{example}
More generally, a {\em directed multi-graph} \cite{ConiglioSernadas}, or {\em directed $m$-graph}, is a pair $(V, A)$ with $V$ a non-empty set of vertices and $A$ a subset of $V^{+}\times V$, where $V^{+}=\bigcup_{n\in\mathbb{N}\setminus\{0\}}V^{n}$ is the set of finite, non-empty, sequences over $V$. We will say that $(v_{1}, \ldots  , v_{n})$ is a {\em path} between $u$ and $v$ if $u=v_{1}$, $v=v_{n}$ and, for every $i\in\{1, \ldots  , n-1\}$, there exists $v_{i_{1}}, \ldots  , v_{i_{m}}$ such that $((v_{i_{1}}, \ldots  , v_{i_{m}}), v_{i+1})$, with $v_{i}=v_{i_{j}}$ for some $j\in\{1, \ldots  , m\}$.

An {\em $m$-forest} is a directed $m$-graph such that any two elements are connected by at most one path, and an $m$-forest is said to have {\em $n$-height} $\omega$, for $n\in\mathbb{N}\setminus\{0\}$, if, for any $(u_{1}, \ldots  , u_{n})\in V^{n}$, there exists $v\in V$ such that $((u_{1}, \ldots  , u_{n}), v)\in A$. Finally, we see that every $m$-forest $F=(V, A)$ with $n$-height $\omega$, for every $n\in S\subseteq\mathbb{N}\setminus\{\emptyset\}$, is essentially equivalent to the $\Sigma_{S}$-multialgebra $\mathcal{A}=(V, \{\sigma_{\mathcal{A}}\}_{\sigma\in\Sigma_{S}})$, with
\begin{equation*}\sigma_{\mathcal{A}}(u_{1}, \ldots  , u_{n})=\{v\in V \ : \ ((u_{1}, \ldots  , u_{n}), v)\in A\},\end{equation*}
for $\sigma$ of arity $n$, and $\Sigma_{S}$ the signature with exactly one operator of arity $n$, for every $n\in S$. It is not hard to see that $\mathcal{A}$ is a submultialgebra of $\textbf{mT}(\Sigma_{S}, \mathcal{V}, |V|)$, with $\mathcal{V}$ the set of elements $v$ of $V$ such that, for no $(u_{1}, \ldots  , u_{n})\in V^{+}$, $((u_{1}, \ldots  , u_{n}), v)\in A$.
\end{example}

\section{Being a submultialgebra of $\textbf{mT}(\Sigma, \mathcal{V}, \kappa)$ as\ldots }\label{Characterizations}

The class of submultialgebras of the members of $\mT(\Sigma, \mathcal{V})=(\textbf{mT}(\Sigma, \mathcal{V}, \kappa))_{\kappa\geq 1}$ is proposed to be the generalization of the free $\Sigma$-algebras to the category  $\textbf{MAlg}(\Sigma)$ of multialgebras. Because of this, the next step is to characterize  the submultialgebras of  $\textbf{mT}(\Sigma, \mathcal{V}, \kappa)$. In this section, five different characterizations of such multialgebras will be found, proving that all of them are equivalent (see Theorem~\ref{charfreemulti}). From this we arrive to the second main notion proposed in this paper: weakly free multialgebras over $\Sigma$.

\subsection{\ldots  being $\textbf{cdf}$-generated}\label{cdf-generated}
 
In universal algebra, the algebras of terms $\textbf{T}(\Sigma, \mathcal{V})$ have the universal mapping property for the class of all $\Sigma$-algebras over $\mathcal{V}$. This means that there exists a set, in their case the set of variables $\mathcal{V}$, such that, for every other $\Sigma$-algebra $\mathcal{B}$ with universe $B$ and function $f:\mathcal{V}\rightarrow B$, there exists a unique homomorphism $\overline{f}:\textbf{T}(\Sigma, \mathcal{V})\rightarrow\mathcal{B}$ extending $f$. As we mentioned before, this is no longer true when dealing with multialgebras, but we can define a closely related concept with the aid of what we will call {\em collections of choices}.

Collections of choices are motivated by the notion of {\em legal valuations}, first defined in Avron and Lev's seminal paper~\cite{AvronLev} in the context of non-deterministic logical matrices. A map $\nu$ from $\textbf{T}(\Sigma, \mathcal{V})$ (seen as the algebra of propositional formulas over $\Sigma$ generated by $\mathcal{V}$) to the universe of a $\Sigma$-multialgebra $\mathcal{A}$ is a legal valuation whenever $\nu(\sigma\alpha_{1} \ldots  \alpha_{n})\in \sigma_{\mathcal{A}}(\nu(\alpha_{1}), \ldots  , \nu(\alpha_{n}))$, for any connective $\sigma$ in $\Sigma$. Essentially, for any formula $\sigma\alpha_{1} \ldots  \alpha_{n}$, $\nu$ ``chooses'' a value from all the possible values $\sigma_{\mathcal{A}}(\nu(\alpha_{1}), \ldots  , \nu(\alpha_{n}))$, possible values which depend themselves on the previous choices $\nu(\alpha_{1})$, \ldots, $\nu(\alpha_{n})$ performed by $\nu$.

A collection of choices automatizes all these aforementioned choices, what justifies its name. 

\begin{definition}
Given multialgebras $\mathcal{A}=(A, \{\sigma_{\mathcal{A}}\}_{\sigma\in\Sigma})$ and $\mathcal{B}=(B, \{\sigma_{\mathcal{B}}\}_{\sigma\in\Sigma})$ over the signature $\Sigma$, a {\it collection of choices} from $\mathcal{A}$ to $\mathcal{B}$ is a collection $C=\{C_{n}\}_{n\in\mathbb{N}}$ of collections of functions
\begin{equation*}C_{n}=\big\{C\sigma_{a_{1}, \ldots  , a_{n}}^{b_{1}, \ldots  , b_{n}} \ : \ \sigma\in\Sigma_{n}, \, a_{1}, \ldots  , a_{n}\in A, \, b_{1}, \ldots   ,b_{n}\in B\big\}\end{equation*}
such that, for $\sigma\in\Sigma_{n}$, $a_{1}, \ldots  , a_{n}\in A$ and $b_{1}, \ldots  , b_{n}\in B$, $C\sigma_{a_{1}, \ldots  , a_{n}}^{b_{1}, \ldots  , b_{n}}$ is a function of the form 
\begin{equation*}C\sigma_{a_{1}, \ldots  , a_{n}}^{b_{1}, \ldots  , b_{n}}:\sigma_{\mathcal{A}}(a_{1}, \ldots   ,a_{n})\rightarrow\sigma_{\mathcal{B}}(b_{1}, \ldots  , b_{n}).\end{equation*}
\end{definition}

\begin{example}
If $\mathcal{B}$ is actually an algebra, that is, all its operations return singletons, there exists only one collection of choices from any $\mathcal{A}$ to $\mathcal{B}$. This means that in the classical environment of universal algebras, collections of choices are somewhat irrelevant.
\end{example}

\begin{example}
A directed tree is a directed forest where there exists exactly one element without predecessor. We say that $v$ ramifies from $u$ if there exists an arrow from $u$ to $v$. Then, for a collection of choices $C$ from $T_{1}$ to $T_{2}$ ($T_{1}=(V_{1}, A_{1})$ and $T_{2}=(V_{2}, A_{2})$ are directed trees of height $\omega$, considered as $\Sigma_{s}$-multialgebras) and for every $v\in V_{1}$ and $u\in V_{2}$, the function $Cs_{v}^{u}$ chooses, for each of the elements that ramify from $v$, one element that ramifies from $u$.
\end{example}

\begin{definition}
Given a signature $\Sigma$, a $\Sigma$-multialgebra $\mathcal{A}=(A, \{\sigma_{\mathcal{A}}\}_{\sigma\in\Sigma})$ is {\it choice-dependent freely generated} by $X$ if $X\subseteq A$ and, for all $\Sigma$-multialgebras $\mathcal{B}=(B, \{\sigma_{\mathcal{B}}\}_{\sigma\in\Sigma})$, all functions $f:X\rightarrow B$ and all collections of choices $C$ from $\mathcal{A}$ to $\mathcal{B}$, there is a unique homomorphism $f_{C}:\mathcal{A}\rightarrow \mathcal{B}$ such that:

\begin{enumerate}
\item $f_{C}|_{X}=f$;

\item for all $\sigma\in \Sigma_{n}$ and $a_{1}, \ldots  , a_{n}\in A$, 
\begin{equation*}f_{C}|_{\sigma_{\mathcal{A}}(a_{1}, \ldots  , a_{n})}=C\sigma_{a_{1}, \ldots  , a_{n}}^{f_{C}(a_{1}), \ldots  , f_{C}(a_{n})}.\end{equation*}
\end{enumerate}
\end{definition}

For simplicity, when $\mathcal{A}$ is choice-dependent freely generated by $X$, we will write that $\mathcal{A}$ is $\textbf{cdf}$-generated by $X$.

In the next definition, we introduce the concept of ground to indicate what elements of a multialgebra are not ``achieved'' by its multioperations. To better visualize this definition one can keep in mind that the ground of an algebra of terms is its set of indecomposable terms, that is, variables.

\begin{definition}
Given a $\Sigma$-multialgebra $\mathcal{A}=(A, \{\sigma_{\mathcal{A}}\}_{\sigma\in\Sigma})$, we define its {\it build} as
%\begin{equation*}B(\mathcal{A})=\bigcup_{n\in\mathbb{N}}\bigcup_{\sigma\in \Sigma_{n}}\bigcup_{(a_{1}, \ldots  , a_{n})\in A^{n}}\sigma_{\mathcal{A}}(a_{1}, \ldots  , a_{n}).\end{equation*}
\begin{equation*}B(\mathcal{A})=\bigcup \big\{\sigma_{\mathcal{A}}(a_{1}, \ldots  , a_{n}) \ : \  n\in\mathbb{N}, \, \sigma\in\Sigma_{n}, \, a_{1}, \ldots  , a_{n}\in A \big\}.\end{equation*}
We define the {\it ground} of $\mathcal{A}$ as
\begin{equation*}G(\mathcal{A})=A\setminus B(\mathcal{A}).\end{equation*}
\end{definition}

\begin{example}\label{ground of formulas}
$B(\textbf{T}(\Sigma, \mathcal{V}))=T(\Sigma, \mathcal{V})\setminus\mathcal{V}$ and $G(\textbf{T}(\Sigma, \mathcal{V}))=\mathcal{V}$.
\end{example}

\begin{example}
If $F=(V, A)$ is a directed forest of height $\omega$, thought as a $\Sigma_{s}$-multialgebra, its ground is the set of elements $v$ in $V$ without predecessors.
\end{example}

\begin{prop}
\begin{enumerate}
\item If $f:\mathcal{A}\rightarrow\mathcal{B}$ is a homomorphism between $\Sigma$-multialgebras, then $B(\mathcal{A})\subseteq f^{-1}(B(\mathcal{B}))$ and $f^{-1}(G(\mathcal{B}))\subseteq G(\mathcal{A})$;

\item If $\mathcal{B}$ is a submultialgebra of $\mathcal{A}$, $B(\mathcal{B})\subseteq B(\mathcal{A})$ and $G(\mathcal{A})\subseteq G(\mathcal{B})$.
\end{enumerate}
\end{prop}

\begin{proof}
\begin{enumerate}
\item If $a\in B(\mathcal{A})$, there exist $\sigma\in\Sigma_{n}$ and $a_{1}, \ldots  , a_{n}\in A$ such that $a\in\sigma_{\mathcal{A}}(a_{1}, \ldots  , a_{n})$. Since
$f(\sigma_{\mathcal{A}}(a_{1}, \ldots  , a_{n})) \subseteq \sigma_{\mathcal{B}}(f(a_{1}), \ldots   ,f(a_{n}))$,
%\begin{equation*}\{f(a) \ : \  a\in\sigma_{\mathcal{A}}(a_{1}, \ldots  , a_{n})\}\subseteq \sigma_{\mathcal{B}}(f(a_{1}), \ldots   ,f(a_{n})),\end{equation*}
we find that $f(a)\in \sigma_{\mathcal{B}}(f(a_{1}), \ldots  , f(a_{n}))$ and therefore $f(a)\in B(\mathcal{B})$, meaning that $a\in f^{-1}(B(\mathcal{B}))$. Using that $G(\mathcal{A})=A\setminus B(\mathcal{A})$ we obtain the second mentioned inclusion.

\item If $b\in B(\mathcal{B})$, there exist $\sigma\in\Sigma_{n}$ and $b_{1}, \ldots  , b_{n}\in B$ such that\\ $b\in \sigma_{\mathcal{B}}(b_{1}, \ldots  , b_{n})$, and given that $\sigma_{\mathcal{B}}(b_{1}, \ldots  , b_{n})\subseteq \sigma_{\mathcal{A}}(b_{1}, \ldots  , b_{n})$ we obtain $b\in B(\mathcal{A})$. Using again that $G(\mathcal{A})=A\setminus B(\mathcal{A})$ we finish the proof. \qedhere
\end{enumerate}
\end{proof}

From this it also follows that if $f:\mathcal{A}\rightarrow\mathcal{B}$ is a homomorphism, $G(\mathcal{B})\cap f(A)$ is contained in $\{f(a)\ : \  a\in G(\mathcal{A})\}$. Indeed,  if $b$ is in $G(\mathcal{B})\cap f(A)$, any $a\in A$ such that $f(a)=b$ is in $f^{-1}(G(\mathcal{B}))$ and, by the previous proposition, is also in $G(\mathcal{A})$. And therefore $b$ is in $\{f(a) \ : \  a\in G(\mathcal{A})\}$.

Generalizing Example \ref{ground of formulas}, we have that $G(\textbf{mT}(\Sigma, \mathcal{V}, \kappa))=\mathcal{V}$, or equivalently $B(\textbf{mT}(\Sigma, \mathcal{V}, \kappa))=T(\Sigma^{\kappa}, \mathcal{V})\setminus \mathcal{V}$, what we show by induction. If $\alpha$ is of order $0$, either we have $\alpha=\sigma^{\beta}$, for a $\sigma\in\Sigma_{0}$ and $\beta\in\kappa$, and therefore $\alpha\in B(\textbf{mT}(\Sigma, \mathcal{V}, \kappa))$; or we have that $\alpha=p\in \mathcal{V}$. In that last case, if there exist $\sigma\in\Sigma_{m}$ and $\alpha_{1}, \ldots  , \alpha_{m}\in T(\Sigma^{\kappa}, \mathcal{V})$ such that 
\begin{equation*}p\in \sigma_{\textbf{mT}(\Sigma, \mathcal{V}, \kappa)}(\alpha_{1}, \ldots  , \alpha_{m}),\end{equation*}
we have $p=\sigma^{\beta}\alpha_{1} \ldots  \alpha_{m}$ for $\beta\in\kappa$, which is absurd given the structure of $T(\Sigma^{\kappa}, \mathcal{V})$, forcing us to conclude that $p\notin B(\textbf{mT}(\Sigma, \mathcal{V}, \kappa))$. If $\alpha$ is of order $n$, we have that $\alpha=\sigma^{\beta}\alpha_{1} \ldots \alpha_{m}$ for $\sigma\in\Sigma_{m}$, $\beta\in\kappa$ and $\alpha_{1}, \ldots  , \alpha_{m}$ of order at most $n-1$, and therefore we have $\alpha$ in $\sigma_{\textbf{mT}(\Sigma, \mathcal{V}, \kappa)}(\alpha_{1}, \ldots  , \alpha_{m})$, meaning that $\alpha\in B(\textbf{mT}(\Sigma, \mathcal{V}, \kappa))$.

\begin{definition}
Given a $\Sigma$-multialgebra $\mathcal{A}=(A, \{\sigma_{\mathcal{A}}\}_{\sigma\in\Sigma})$ and a set $S\subseteq A$, we define the sets $\langle S\rangle_{m}$ by induction: $\langle S\rangle _{0}=S\cup\bigcup_{\sigma\in \Sigma_{0}}\sigma_{\mathcal{A}}$; and assuming we have defined $\langle S\rangle_{m}$, we define
%\[\langle S\rangle_{m+1}=\langle S\rangle_{m}\cup\bigcup_{n\in\mathbb{N}}\bigcup_{\sigma\in\Sigma_{n}}\bigcup_{(a_{1}, \ldots  , a_{n})\in\langle S\rangle_{m}^{n}}\sigma_{\mathcal{A}}(a_{1}, \ldots   ,a_{n}).\]
\[\langle S\rangle_{m+1}=\langle S\rangle_{m}\cup \bigcup \big\{\sigma_{\mathcal{A}}(a_{1}, \ldots  , a_{n}) \ : \  n\in\mathbb{N}, \, \sigma\in\Sigma_{n}, \, a_{1}, \ldots  , a_{n}\in \langle S\rangle_{m} \big\}.\]
The {\it set generated} by $S$, denoted by $\langle S\rangle$, is then defined as $\langle S\rangle=\bigcup_{m\in\mathbb{N}}\langle S\rangle_{m}$.

We say $\mathcal{A}$ is {\it generated} by $S$ if $\langle S\rangle=A$.
\end{definition}

\begin{lemma}\label{sub mT is gen ground}
Every submultialgebra $\mathcal{A}$ of $\textbf{mT}(\Sigma, \mathcal{V}, \kappa)$ is generated by $G(\mathcal{A})$.
\end{lemma}

\begin{proof}
Suppose $a$ is an element of $\mathcal{A}$ not contained in $\langle G(\mathcal{A})\rangle$ of minimum order. Given that $a$ cannot belong to $G(\mathcal{A})\cup\bigcup_{\sigma\in\Sigma_{0}}\sigma_{\mathcal{A}}=\langle G(\mathcal{A})\rangle_{0}$, there exist $n>0$, $\sigma\in\Sigma_{n}$ and $a_{1}, \ldots  , a_{n}\in A$ such that $a\in\sigma_{\mathcal{A}}(a_{1}, \ldots  , a_{n})$.

Since $\sigma_{\mathcal{A}}(a_{1}, \ldots  , a_{n})\subseteq \sigma_{\textbf{mT}(\Sigma, X, \kappa)}(a_{1}, \ldots  , a_{n})$ we derive that $a_{1}, \ldots,a_{n}$ are of smaller order than $a$. By our hypothesis, there must exist $m_{1}, \ldots  , m_{n}$ such that $a_{j}\in \langle G(\mathcal{A})\rangle_{m_{j}}$ for all $j\in\{1, \ldots  , n\}$; taking $m=\max\{m_{1}, \ldots  , m_{n}\}$ one obtains that $a_{1}, \ldots  , a_{n}\in \langle G(\mathcal{A})\rangle_{m}$, and therefore
\begin{equation*}a\in \sigma_{\mathcal{A}}(a_{1}, \ldots  , a_{n})\subseteq \langle G(\mathcal{A})\rangle_{m+1},\end{equation*}
which contradicts our assumption that $a$ is not in $\langle G(\mathcal{A})\rangle$. 
\end{proof}

\begin{theorem}\label{sub mT is cdf-gen ground}
Every submultialgebra $\mathcal{A}$ of $\textbf{mT}(\Sigma, \mathcal{V}, \kappa)$ is $\textbf{cdf}$-generated by $G(\mathcal{A})$.
\end{theorem}

\begin{proof}
Let $\mathcal{A}=(A, \{\sigma_{\mathcal{A}}\}_{\Sigma})$ be a submultialgebra of $\textbf{mT}(\Sigma, \mathcal{V}, \kappa)$, let $\mathcal{B}=(B, \{\sigma_{\mathcal{B}}\}_{\Sigma})$ be any $\Sigma$-multialgebra, let $f:G(\mathcal{A})\rightarrow B$ be a function and $C$ a collection of choices from $\mathcal{A}$ to $\mathcal{B}$. We define $f_{C}:\mathcal{A}\rightarrow\mathcal{B}$ by induction on $\langle G(\mathcal{A})\rangle_{m}$:

\begin{enumerate}
\item if $a\in \langle G(\mathcal{A})\rangle_{0}$ and $a\in G(\mathcal{A})$, we define $f_{C}(a)=f(a)$;

\item if $a\in \langle G(\mathcal{A})\rangle_{0}$ and $a\in\sigma_{\mathcal{A}}$, for some $\sigma\in\Sigma_{0}$, we define $f_{C}(a)=C\sigma(a)$;

\item if $f_{C}$ is defined for all elements of $\langle G(\mathcal{A})\rangle_{m}$, $a_{1}, \ldots  , a_{n}\in \langle G(\mathcal{A})\rangle_{m}$ and $\sigma\in\Sigma_{n}$, for every element $a\in \sigma_{\mathcal{A}}(a_{1}, \ldots  , a_{n})$ we define 
\begin{equation*}f_{C}(a)=C\sigma_{a_{1}, \ldots  , a_{n}}^{f_{C}(a_{1}), \ldots  , f_{C}(a_{n})}(a).\end{equation*}
\end{enumerate}

First, we must prove that $f_{C}$ is well defined. There are two possibly problematic cases to consider for an element $a\in A$: 
\begin{enumerate}
\item the one in which $a\in G(\mathcal{A})$ and there are $\sigma\in\Sigma_{n}$ and $a_{1}, \ldots  , a_{n}\in A$ with $a\in\sigma_{\mathcal{A}}(a_{1}, \ldots  , a_{n})$, corresponding to $a$ falling simultaneously in the cases $(1)$ and $(2)$, or $(1)$ and $(3)$ of the definition; 
\item and the one where there are $\sigma\in\Sigma_{n}$, $\theta\in\Sigma_{m}$, $a_{1}, \ldots  , a_{n}\in A$ and $b_{1}, \ldots  , b_{m}\in A$ such that $a\in\sigma_{\mathcal{A}}(a_{1}, \ldots  , a_{n})$ and $a\in\theta_{\mathcal{A}}(b_{1}, \ldots  , b_{m})$, a situation that corresponds to the cases $(2)$ and $(3)$, $(2)$ and $(2)$,\footnote{That is, $a\in\langle G(\mathcal{A})\rangle_{0}$, and $a\in\sigma_{\mathcal{A}}$ and $a\in\theta_{\mathcal{A}}$, for different $\sigma, \theta\in \Sigma_{0}$, where defining $f_{C}(a)$ as both $C\sigma(a)$ and $C\theta(a)$ could be impossible.} or $(3)$ and $(3)$\footnote{That is, $f_{C}$ is defined for all of $\langle G(\mathcal{A})\rangle_{k}$, $a_{1}, \ldots  , a_{n}, b_{1}, \ldots  , b_{m}\in\langle G(\mathcal{A})\rangle_{k}$, and $a\in\sigma_{\mathcal{A}}(a_{1}, \ldots  , a_{n})$ and $a\in\theta_{\mathcal{A}}(b_{1}, \ldots  , b_{m})$, for $\sigma\in\Sigma_{n}$ and $\theta\in\Sigma_{m}$, meaning it could be impossible to define $f_{C}(a)$ in a systematic way.} occurring simultaneously.
\end{enumerate}

The first case is not possible, since $G(\mathcal{A})\subseteq A\setminus\sigma_{\mathcal{A}}(a_{1}, \ldots  , a_{n})$ for every $\sigma\in\Sigma_{n}$ and $a_{1}, \ldots  , a_{n}\in A$. In the second case, we find that 
\begin{equation*}a\in \sigma_{\mathcal{A}}(a_{1}, \ldots  , a_{n})\cap\theta_{\mathcal{A}}(b_{1}, \ldots  , b_{m})\subseteq\sigma_{\textbf{mT}(\Sigma, \mathcal{V}, \kappa)}(a_{1}, \ldots  , a_{n})\cap \theta_{\textbf{mT}(\Sigma, \mathcal{V}, \kappa)}( b_{1}, \ldots   , b_{m}),\end{equation*}
so $n=m$, $\sigma=\theta$ and $a_{1}=b_{1}, \ldots  , a_{n}=b_{m}$, and therefore $f_{C}(a)$ is well-defined.

Second, we must prove that $f_{C}$ is defined over all of $A$. That is simple, for $f_{C}$ is defined over all of $\langle G(\mathcal{A})\rangle$ and we established in Lemma~\ref{sub mT is gen ground} that $A=\langle G(\mathcal{A})\rangle$.

So $f_{C}:A\rightarrow B$ is a well-defined function. It remains to be shown that it is a homomorphism. So,  given $\sigma\in\Sigma_{n}$ and $a_{1}, \ldots  , a_{n}$, we see that
\begin{equation*}f_{C}(\sigma_{\mathcal{A}}(a_{1}, \ldots  , a_{n}))=\big\{C\sigma_{a_{1}, \ldots  , a_{n}}^{f_{C}(a_{1}), \ldots  , f_{C}(a_{n})}(a)\ : \  a\in\sigma_{\mathcal{A}}(a_{1}, \ldots  , a_{n})\big\}\subseteq\sigma_{\mathcal{B}}(f_{C}(a_{1}), \ldots  , f_{C}(a_{n})),\end{equation*}
while we also have that $f_{C}$ clearly extends both $f$ and all $C\sigma_{a_{1}, \ldots  , a_{n}}^{f_{C}(a_{1}, \ldots  , f_{C}(a_{n})}$. 

To finish the proof, suppose $g:\mathcal{A}\rightarrow\mathcal{B}$ is another homomorphism extending both $f$ and all $C\sigma_{a_{1}, \ldots  , a_{n}}^{g(a_{1}, \ldots  , g(a_{n})}$. We will prove that $g=f_{C}$ again by induction on the $m$ of $\langle G(\mathcal{A})\rangle_{m}$. For $m=0$, an element $a\in \langle G(\mathcal{A})\rangle_{0}$ is either in $G(\mathcal{A})$, when we have $g(a)=f(a)=f_{C}(a)$, or in $\sigma_{\mathcal{A}}$ for a $\sigma\in\Sigma_{0}$, when $g(a)=C\sigma(a)=f_{C}(a)$.

Suppose $g$ is equal to $f_{C}$ in $\langle G(\mathcal{A})\rangle_{m}$ and take an $a\in \langle G(\mathcal{A})\rangle_{m+1}\setminus \langle G(\mathcal{A})\rangle_{m}$. Then, there exist $\sigma\in\Sigma_{n}$ and $a_{1}, \ldots  , a_{n}\in \langle G(\mathcal{A})\rangle_{m}$ such that $a\in\sigma_{\mathcal{A}}(a_{1}, \ldots  , a_{n})$ and so
\begin{equation*}g(a)=C\sigma_{a_{1}, \ldots  , a_{n}}^{g(a_{1}), \ldots  , g(a_{n})}(a)=C\sigma_{a_{1}, \ldots  , a_{n}}^{f_{C}(a_{1}), \ldots  , f_{C}(a_{n})}(a)=f_{C}(a).\end{equation*}
This proves that $g=f_{C}$ and that, in fact, $f_{C}$ is unique. That is, $\mathcal{A}$ is $\textbf{cdf}$-generated by $G(\mathcal{A})$.
\end{proof}

The proof of the following lemma may be found in Section 2 of \cite{CFG}.

\begin{lemma}\label{image of hom is sub}
Let $\mathcal{A}=(A, \{\sigma_{\mathcal{A}}\}_{\sigma\in\Sigma})$ and $\mathcal{B}=(B, \{\sigma_{\mathcal{B}}\}_{\sigma\in\Sigma})$ be $\Sigma$-multialgebras, and let $f:\mathcal{A}\rightarrow\mathcal{B}$ be a homomorphism. Then, the structure $\mathcal{C}=(f(A), \{\sigma_{\mathcal{C}}\}_{\sigma\in\Sigma})$ such that
\[\sigma_{\mathcal{C}}(c_{1}, \ldots , c_{n})=\bigcup\{f(\sigma_{\mathcal{A}}(a_{1}, \ldots , a_{n})) : f(a_{1})=c_{1}, \ldots , f(a_{n})=c_{n}\}\]
is a $\Sigma$-submultialgebra of $\mathcal{B}$, while $f:\mathcal{A}\rightarrow\mathcal{C}$ is an epimorphism. The $\Sigma$-multialgebra $\mathcal{C}$ is known as the \textit{direct image} of $\mathcal{A}$ trough $f$.
\end{lemma}

\begin{theorem}\label{cdf-gen is iso to sub mT}
If the multialgebra $\mathcal{A}=(A, \{\sigma_{\mathcal{A}}\}_{\sigma\in\Sigma})$ over $\Sigma$ is $\textbf{cdf}$-generated by $X$, then $\mathcal{A}$ is isomorphic to a submultialgebra of $\textbf{mT}(\Sigma, X, |A|)$ containing $X$.
\end{theorem}

\begin{proof}
Take $f:X\rightarrow T(\Sigma^{|A|}, X)$ to be the identity map (that is, $f(x)=x$), and take a collection of choices $C$ such that, for $\sigma\in\Sigma_{n}$, $a_{1}, \ldots  , a_{n}\in A$ and $\alpha_{1}, \ldots  , \alpha_{n}\in T(\Sigma^{|A|}, X)$, 
\begin{equation*}C\sigma_{a_{1}, \ldots  , a_{n}}^{\alpha_{1}, \ldots  , \alpha_{n}}:\sigma_{\mathcal{A}}(a_{1}, \ldots  , a_{n})\rightarrow\sigma_{\textbf{mT}(\Sigma, X, |A|)}(\alpha_{1}, \ldots  , \alpha_{n})\end{equation*}
is an injective function. Such collection of choices exist since $\sigma_{\mathcal{A}}(a_{1}, \ldots  , a_{n})\subseteq A$ and $\sigma_{\textbf{mT}(\Sigma, X, |A|)}(\alpha_{1}, \ldots  , \alpha_{n})$ is of cardinality $|A|$. Since $\mathcal{A}$ is $\textbf{cdf}$-generated by $X$, there exists a homomorphism $f_{C}:\mathcal{A}\rightarrow\textbf{mT}(\Sigma, X, |A|)$ extending $f$ and each $C\sigma_{a_{1}, \ldots  , a_{n}}^{f_{C}(a_{1}), \ldots  , f_{C}(a_{n})}$.

Let $\mathcal{B}=(f_{C}(A), \{\sigma_{\mathcal{B}}\}_{\sigma\in\Sigma})$ be the direct image of $\mathcal{A}$ trough $f_{C}$, so that $f_{C}:\mathcal{A}\rightarrow\mathcal{B}$ is an epimorphism, what is possible given Lemma~\ref{image of hom is sub}. Notice too that 
\begin{equation*}X=X\cap f_{C}(A)=G(\textbf{mT}(\Sigma, X, |A|))\cap f_{C}(A)\subseteq G(\mathcal{B})\end{equation*}
because $\mathcal{B}$ is a submultialgebra of $\textbf{mT}(\Sigma, X, |A|)$. Now, take any $g:G(\mathcal{B})\rightarrow A$ such that $g(x)=x$, for every $x\in X$, and a collection of choices $D$ from $\mathcal{B}$ to $\mathcal{A}$ such that, for any $\sigma\in\Sigma_{n}$, $b_{1}, \ldots  , b_{n}\in f_{C}(A)$ and $a_{1}, \ldots  , a_{n}\in A$, the function
\begin{equation*}D\sigma_{b_{1}, \ldots  , b_{n}}^{a_{1}, \ldots  , a_{n}}:\sigma_{\mathcal{B}}(b_{1}, \ldots  , b_{n})\rightarrow\sigma_{\mathcal{A}}(a_{1}, \ldots  , a_{n})\end{equation*}
satisfies the following: if $a \in \sigma_{\mathcal{A}}(a_{1}, \ldots  , a_{n})$ is such that $C\sigma_{a_{1}, \ldots , a_{n}}^{b_{1}, \ldots , b_{n}}(a) \in \sigma_{\mathcal{B}}(b_{1}, \ldots , b_{n})$ then\\ $D\sigma_{b_{1}, \ldots  , b_{n}}^{a_{1}, \ldots  , a_{n}}(C\sigma_{a_{1}, \ldots  , a_{n}}^{b_{1}, \ldots  , b_{n}}(a))=a$. Given that $C\sigma_{a_{1}, \ldots  , a_{n}}^{b_{1}, \ldots  , b_{n}}$ is injective, this condition is well-defined.

Since $\mathcal{B}$ is $\textbf{cdf}$-generated by $G(\mathcal{B})$, we know to exist a homomorphism $g_{D}:\mathcal{B}\rightarrow \mathcal{A}$ extending $g$ and the functions $D\sigma_{b_{1}, \ldots  , b_{n}}^{g_{D}(b_{1}), \ldots  , g_{D}(b_{n})}$. 

Finally, we take $g_{D}\circ f_{C}:\mathcal{A}\rightarrow\mathcal{A}$. It extends the injection $id=g\circ f:X\rightarrow A$, for which $id(x)=x$. It also extends the collection of choices $E$ defined by
\begin{equation*}E\sigma_{a_{1}, \ldots , a_{n}}^{a'_{1}, \ldots , a'_{n}}=D\sigma_{f_{C}(a_{1}), \ldots , f_{C}(a_{n})}^{a'_{1}, \ldots , a'_{n}}\circ C_{a_{1}, \ldots , a_{n}}^{f_{C}(a_{1}), \ldots , f_{C}(a_{n})}:\sigma_{\mathcal{A}}(a_{1}, \ldots , a_{n})\rightarrow\sigma_{\mathcal{A}}(a'_{1}, \ldots , a'_{n}),\end{equation*}
for $\sigma\in\Sigma_{n}$ and $a_{1}, \ldots a_{n}, a'_{1}, \ldots , a'_{n}\in A$. This way, $E\sigma_{a_{1}, \ldots , a_{n}}^{a_{1}, \ldots , a_{n}}$ is the identity on $\sigma_{\mathcal{A}}(a_{1}, \ldots , a_{n})$. Indeed, for any $a\in \sigma_{\mathcal{A}}(a_{1}, \ldots , a_{n})$, 
\begin{equation*}C\sigma_{a_{1}, \ldots , a_{n}}^{f_{C}(a_{1}), \ldots , f_{C}(a_{n})}(a)=f_{C}(a)\end{equation*}
by definition of $f_{C}$, and, given that $f_{C}:\mathcal{A}\rightarrow\mathcal{B}$ is a homomorphism, $f_{C}(a)$ belongs to $\sigma_{\mathcal{B}}(f_{C}(a_{1}, \ldots , f_{C}(a_{n}))$, meaning that $C\sigma_{a_{1}, \ldots , a_{n}}^{f_{C}(a_{1}), \ldots , f_{C}(a_{n})}(a) \in \sigma_{\mathcal{B}}(f_{C}(a_{1}), \ldots , f_{C}(a_{n}))$. Then
\begin{equation*}E\sigma_{a_{1}, \ldots , a_{n}}^{a_{1}, \ldots , a_{n}}(a) = D\sigma_{f_{C}(a_{1}), \ldots , f_{C}(a_{n})}^{a_{1}, \ldots , a_{n}}(C_{a_{1}, \ldots , a_{n}}^{f_{C}(a_{1}), \ldots , f_{C}(a_{n})}(a)) = a\end{equation*}
by definition of $D$.

%NÃO APAGUEI O TEXTO QUE HAVIA AQUI ANTES, ELE SE ENCONTRA APOS O FINAL DO DOCUMENTO SE VOCE PRECISAR

But notice that the identical homomorphism $\mathcal{I}:\mathcal{A}\rightarrow \mathcal{A}$ also extends both $id$ and $E$ and, given the unicity of such extensions on the definition of being $\textbf{cdf}$-generated, we obtain that $\mathcal{I}=g_{D}\circ f_{C}$. The fact that $f_{C}:\mathcal{A}\rightarrow\mathcal{B}$ has a left inverse implies that it is injective, and by definition of $\mathcal{B}$ it is also surjective, meaning that it is a bijective function.
Moreover, $g_{D}$ is the inverse function of $f_{C}$.
Finally, for $\sigma\in \Sigma_{n}$ and $a_{1}, \ldots , a_{n}\in A$,
\begin{equation*}f_{C}(\sigma_{\mathcal{A}}(a_{1}, \ldots , a_{n}))\subseteq \sigma_{\mathcal{B}}(f_{C}(a_{1}), \ldots , f_{C}(a_{n})),\end{equation*}
since $f_{C}$ is a homomorphism. However, given that $g_{D}$ is also a homomorphism,
\begin{equation*}g_{D}(\sigma_{\mathcal{B}}(f_{C}(a_{1}), \ldots , f_{C}(a_{n})))\subseteq \sigma_{\mathcal{A}}(g_{D}\circ f_{C}(a_{1}), \ldots , g_{D}\circ f_{C}(a_{n}))=\sigma_{\mathcal{A}}(a_{1}, \ldots , a_{n}),\end{equation*}
and by applying $f_{C}$ to both sides, one obtains
\begin{equation*}\sigma_{\mathcal{B}}(f_{C}(a_{1}), \ldots , f_{C}(a_{n}))=f_{C}(g_{D}(\sigma_{\mathcal{B}}(f_{C}(a_{1}), \ldots , f_{C}(a_{n}))))\subseteq f_{C}(\sigma_{\mathcal{A}}(a_{1}, \ldots , a_{n})).\end{equation*}
This proves that $f_{C}$ is a bijective full homomorphism, that is, an isomorphism.
\end{proof}

Notice that, from the proof above, we can see that if $\mathcal{A}=(A, \{\sigma_{\mathcal{A}}\}_{\sigma\in\Sigma})$ is $\textbf{cdf}$-generated by $X$, then $\mathcal{A}$ is in fact isomorphic to a submultialgebra of $\textbf{mT}(\Sigma, X, M(\mathcal{A}))$, where
\begin{equation*}M(\mathcal{A})=\max \big\{|\sigma_{\mathcal{A}}(a_{1}, \ldots  , a_{n})| \ : \  n\in\mathbb{N}, \, \sigma\in\Sigma_{n}, \, a_{1}, \ldots  , a_{n}\in A \big\}.\end{equation*}
It is clear that $M(\mathcal{A})=\kappa$ for the multialgebra $\textbf{mT}(\Sigma, \mathcal{V}, \kappa)$. The value $M(\mathcal{A})$ has been already regarded in the literature as an important aspect of multialgebras, see~\cite{CuponaMadarasz} (observe, however, that their definition of homomorphism is quite different from ours).

Notice, furthermore, that written in classical terms, the previous Theorems~\ref{sub mT is cdf-gen ground} and~\ref{cdf-gen is iso to sub mT} state a well known result: an algebra is absolutely free iff it is isomorphic to some algebra of terms over the same signature.

\begin{cor}\label{cdf-gen is cdf-gen by ground}
Every $\textbf{cdf}$-generated multialgebra $\mathcal{A}$ is generated by its ground $G(\mathcal{A})$.
\end{cor}

\begin{proof}
Since every $\textbf{cdf}$-generated multialgebra is isomorphic to a submultialgebra of some $\textbf{mT}(\Sigma, X, \kappa)$, from \ref{cdf-gen is iso to sub mT}, and every submultialgebra of $\textbf{mT}(\Sigma, X, \kappa)$ is generated by its ground, the result follows.
\end{proof}

\begin{cor}
Every $\textbf{cdf}$-generated multialgebra $\mathcal{A}$ is $\textbf{cdf}$-generated by its ground $G(\mathcal{A})$.
\end{cor}

\begin{definition}
A $\Sigma$-multialgebra $\mathcal{A}=(A, \{\sigma_{\mathcal{A}}\}_{\sigma\in\Sigma})$ is said to be {\it disconnected} if, for every $\sigma\in\Sigma_{n}, \theta\in \Sigma_{m}$, $a_{1}, \ldots   , a_{n}, b_{1}, \ldots   , b_{m}\in A$,
\begin{equation*}\sigma_{\mathcal{A}}(a_{1}, \ldots  , a_{n})\cap \theta_{\mathcal{A}}(b_{1}, \ldots  , b_{m})\neq\emptyset\end{equation*}
implies that $n=m$, $\sigma=\theta$ and $a_{1}=b_{1}, \ldots  , a_{n}=b_{m}$.
\end{definition}

\begin{example}
$\textbf{T}(\Sigma, \mathcal{V})$ is disconnected.
\end{example}

\begin{example}
All directed forests of height $\omega$, when considered as $\Sigma_{s}$-multialgebras, are disconnected, given that no two arrows point to the same element.
\end{example}

It is clear that if $\mathcal{B}$ is a submultialgebra of $\mathcal{A}$ and $\mathcal{A}$ is disconnected, then $\mathcal{B}$ is also disconnected, since if $\sigma_{\mathcal{B}}(a_{1}, \ldots  , a_{n})\cap \theta_{\mathcal{B}}(b_{1}, \ldots  , b_{m})\neq\emptyset$, for $a_{1}, \ldots  , a_{n}, b_{1}, \ldots  , b_{m}\in B$, given that $\sigma_{\mathcal{B}}(a_{1}, \ldots  , a_{n})\subseteq\sigma_{\mathcal{A}}(a_{1}, \ldots  , a_{n})$ and $\theta_{\mathcal{B}}(b_{1}, \ldots  , b_{m})\subseteq \theta_{\mathcal{A}}(b_{1}, \ldots  , b_{m})$, we find that $\sigma_{\mathcal{A}}(a_{1}, \ldots  , a_{n})\cap \theta_{\mathcal{A}}(b_{1}, \ldots  , b_{m})\neq\emptyset$ and therefore $n=m$, $\sigma=\theta$ and $a_{1}=b_{1}, \ldots  , a_{n}=b_{m}$.

We noticed before that $\textbf{mT}(\Sigma, \mathcal{V}, \kappa)$ is disconnected, and by Theorem~\ref{cdf-gen is iso to sub mT} we obtain that every $\textbf{cdf}-$generated algebra is disconnected. This also means something deeper: being disconnected is, in a way, a measure of how free of identities a multialgebra is. After all, the fact that no two multioperations agree on any elements is strongly indicative that the multialgebra does not satisfy any identities.

\subsection{\ldots being disconnected and generated by its ground}\label{disc. and ground}

Now, we continue to look at other possible characterizations of multialgebras of terms that could lead to a generalization of free algebras (i.e., of algebras of terms). One sees that algebras of terms do not have identities, what would partially correspond in our study to the concept of being disconnected. But what is possibly more representative of our intuition for terms is that one starts by defining them from elements that are as simple as possible (variables), and continues indefinitely. The concept of indecomposable element is here replaced by that of being an element of the ground, so one would expect that being generated by it plays some role in what we have defined so far.

\begin{lemma}\label{X in ground if cdf-gen by X}
If $\mathcal{A}$ is $\textbf{cdf}$-generated by $X$, then $X\subseteq G(\mathcal{A})$.
\end{lemma}

\begin{proof}
If $\mathcal{A}$ is $\textbf{cdf}$-generated by $X$, then $\mathcal{A}$ is isomorphic to a submultialgebra of $\textbf{mT}(\Sigma, X, |A|)$ containing $X$, from Theorem~\ref{cdf-gen is iso to sub mT}. Let us assume that $\mathcal{A}$ is equal to this submultialgebra, without loss of generality.
Then,  $X=G(\textbf{mT}(\Sigma, X, |A|))\cap A\subseteq G(\mathcal{A})$.
\end{proof}

\begin{lemma}\label{no proper cdf-gen sets}
If $\mathcal{A}$ is $\textbf{cdf}$-generated by both $X$ and $Y$, with $X\subseteq Y$, then $X=Y$.
\end{lemma}

\begin{proof}
Suppose $X\neq Y$ and let $y\in Y\setminus X$. Take a $\Sigma$-multialgebra $\mathcal{B}$, over the same signature as that of $\mathcal{A}$, such that $|B|\geq 2$, and a collection of choices $C$ from $\mathcal{A}$ to $\mathcal{B}$.

Take also two functions $g, h:Y\rightarrow B$ such that $g|_{X}=h|_{X}=f$ and $g(y)\neq h(y)$, what is possible since $|B|\geq 2$. Given that $\mathcal{A}$ is $\textbf{cdf}$-generated by $Y$, there exist unique homomorphisms $g_{C}$ and $h_{C}$ extending both $g$ and $C$, and $h$ and $C$, respectively.

However, $g_{C}$ and $h_{C}$ extend both $g|_{X}:X\rightarrow B$ and $C$, and since $\mathcal{A}$ is $\textbf{cdf}$-generated by $X$, we find that $g_{C}=h_{C}$. This is not possible, since $g_{C}(y)\neq h_{C}(y)$, what must implies that $Y\setminus X=\emptyset$ and therefore $X=Y$.
\end{proof}

\begin{theorem}\label{ground is only cdf-gen set}
Every $\textbf{cdf}$-generated multialgebra $\mathcal{A}$ is uniquely $\textbf{cdf}$-generated by its ground.
\end{theorem}

\begin{proof}
From Corollary~\ref{cdf-gen is cdf-gen by ground}, $\mathcal{A}$ is $\textbf{cdf}$-generated by $G(\mathcal{A})$, and from Lemma~\ref{X in ground if cdf-gen by X}, if $\mathcal{A}$ is also $\textbf{cdf}$-generated by $X$, then $X\subseteq G(\mathcal{A})$. By Lemma~\ref{no proper cdf-gen sets}, this implies that $X=G(\mathcal{A})$.
\end{proof}

We have proved so far that if $\mathcal{A}$ is $\textbf{cdf}$-generated, then $\mathcal{A}$ is generated by its ground and disconnected. We would like to prove that this is enough to characterize a $\textbf{cdf}$-generated multialgebra. That is, if $\mathcal{A}$ is generated by its ground and disconnected, then it is $\textbf{cdf}$-generated, exactly by its ground.

The idea is similar to the one we used to prove that all submultialgebras of $\textbf{mT}(\Sigma, \mathcal{V}, \kappa)$ are $\textbf{cdf}$-generated: take a multialgebra $\mathcal{A}$ that is both generated by its ground $G(\mathcal{A})$, which will be denoted by $X$, and disconnected, and fix a multialgebra $\mathcal{B}$ over the same signature, a function $f:X\rightarrow B$ and a collection of choices $C$ from $\mathcal{A}$ to $\mathcal{B}$.

We define a function $f_{C}:A\rightarrow B$ using induction on the $\langle X\rangle_{n}$. For $n=0$, either we have an element $x\in X$, when we define $f_{C}(x)=f(x)$, or we have $a\in\sigma_{\mathcal{A}}$ for some $\sigma\in\Sigma_{0}$, when we define $f_{C}(a)=C\sigma(a)$. Notice that, up to this point, there are no contradictions in this definition, given that an element cannot belong both to $X$ and to a $\sigma_{\mathcal{A}}$, since $X=G(\mathcal{A})$.

Suppose we have successfully defined $f_{C}$ on $\langle X\rangle_{m}$ and take an $a\in \sigma_{\mathcal{A}}(a_{1}, \ldots  , a_{n})$ for $a_{1}, \ldots  , a_{n}\in\langle X\rangle_{m}$. We then define 
\begin{equation*}f_{C}(a)=C\sigma_{a_{1}, \ldots  , a_{n}}^{f_{C}(a_{1}), \ldots  , f_{C}(a_{n})}(a).\end{equation*}
Again the function remains well-defined: $a$ cannot belong to $X$, since $X=G(\mathcal{A})$, and cannot belong to a $\theta_{\mathcal{A}}(b_{1}, \ldots  , b_{p})$ unless $p=n$, $\theta=\sigma$ and $b_{1}=a_{1}, \ldots   ,b_{p}=a_{n}$, since $\mathcal{A}$ is disconnected.

Clearly $f_{C}$ is a homomorphism, since the image of $\sigma_{\mathcal{A}}(a_{1}, \ldots  , a_{n})$ under $f_{C}$ is contained in $\sigma_{\mathcal{B}}(f_{C}(a_{1}), \ldots  , f_{C}(a_{n}))$, and $f_{C}$ extends both $f$ and $C$.

\begin{lemma}\label{gen by ground and disc implies cdf-gen}
If a multialgebra $\mathcal{A}$ is both generated by its ground $X$ and disconnected, $\mathcal{A}$ is $\textbf{cdf}$-generated by $X$.
\end{lemma}

\begin{proof}
It remains for us to show that $f_{C}$, as defined above, is the only homomorphism extending $f$ and $C$. Suppose $g$ is another such homomorphism and we shall proceed yet again by induction.

On $\langle X\rangle_{0}$, we have that $f_{C}(x)=f(x)=g(x)$ for all $x\in X$; and for $\sigma\in\Sigma_{0}$ and $a\in\sigma_{\mathcal{A}}$ we have that 
\begin{equation*}f_{C}(a)=C\sigma(a)=g(a),\end{equation*}
hence $f_{C}$ and $g$ coincide on $\langle X\rangle_{0}$. Suppose that $f_{C}$ and $g$ are equal on $\langle X\rangle_{m}$ and take $a\in\sigma_{\mathcal{A}}(a_{1}, \ldots  , a_{n})$ for $a_{1}, \ldots  , a_{n}\in \langle X\rangle_{m}$. We have by induction hypothesis that
\begin{equation*}f_{C}(a)=C\sigma_{a_{1}, \ldots  , a_{n}}^{f_{C}(a_{1}), \ldots  , f_{C}(a_{n})}(a)=C\sigma_{a_{1}, \ldots  , a_{n}}^{g(a_{1}), \ldots  , g(a_{n})}(a)=g(a),\end{equation*}
which concludes our proof.
\end{proof}

\begin{theorem}\label{cdf-gen iff gen by ground and disc}
A multialgebra $\mathcal{A}$ is $\textbf{cdf}$-generated iff $\mathcal{A}$ is generated by its ground and disconnected.
\end{theorem}

It is important to analyze, by means of examples,  the differences between the several concepts involved: are there multialgebras that are disconnected but not generated by their grounds? Are there multialgebras that are generated by their grounds but not disconnected? If not, does being generated by its ground imply being disconnected or vice-versa? We show below that this is not the case by providing examples answering positively both previous questions.

\begin{example}\label{C}
Take the signature $\Sigma_{s}$ from Example~\ref{s}. Consider the $\Sigma_{s}-$multialgebra $\mathcal{C}=(\{-1,1\}, \{s_{\mathcal{C}}\})$ such that $s_{\mathcal{C}}(-1)=\{1\}$ and $s_{\mathcal{C}}(1)=\{-1\}$ (that is, $s_{\mathcal{C}}(x)=\{-x\}$).

We state that $\mathcal{C}$ is disconnected, but not generated by its ground. $\mathcal{C}$ is clearly disconnected since $s_{\mathcal{C}}(-1)\cap s_{\mathcal{C}}(1)=\emptyset$; now, $B(\mathcal{C})=s_{\mathcal{C}}(-1)\cup s_{\mathcal{C}}(1)=\{-1,1\}$, and so $G(\mathcal{C})=\emptyset$.

Since $\Sigma_s$ has no $0$-ary operators and $G(\mathcal{C})=\emptyset$, it follows that $\langle G(\mathcal{C})\rangle_{0}=\emptyset$
 and therefore $\langle G(\mathcal{C})\rangle_{n}=\emptyset$ for every $n\in\mathbb{N}$, meaning that $G(\mathcal{C})$ does not generate $\mathcal{C}$.
\end{example}

\begin{example}\label{B}
Take again the signature $\Sigma_{s}$ with a single unary operator, from Example~\ref{s}. Consider the $\Sigma_{s}$-multialgebra $\mathcal{B}=(\{0, 1\}, \{s_{\mathcal{B}}\})$ such that $s_{\mathcal{B}}(0)=\{1\}$ and $s_{\mathcal{B}}(1)=\{1\}$ (that is, $s_{\mathcal{B}}(x)=\{1\}$).

Then $\mathcal{B}$ is clearly not disconnected, since $s_{\mathcal{B}}(0)\cap s_{\mathcal{B}}(1)=\{1\}$, yet $\mathcal{B}$ is generated by its ground: $B(\mathcal{B})=\{1\}$ and so $G(\mathcal{B})=\{0\}$, and we see that $\langle G(\mathcal{B})\rangle_{1}$ is already $\{0,1\}$.
\end{example}

\begin{figure}[H]
\centering
\begin{minipage}[t]{4cm}
\centering
\begin{tikzcd}
    -1 \arrow[rr, bend left=50, "s_{\mathcal{C}}"]  && 1 \arrow[ll, bend left=50, "s_{\mathcal{C}}"]
  \end{tikzcd}
\caption*{The $\Sigma_{s}$-multialgebra $\mathcal{C}$}
\end{minipage}
\hspace{3cm}
\centering
\begin{minipage}[t]{4cm}
\centering
\begin{tikzcd}
    0 \arrow[r, "s_{\mathcal{B}}"]  & 1 \arrow[loop right, out=30, in=-30, distance=3em]{}{s_{\mathcal{B}}}
  \end{tikzcd}
\caption*{The $\Sigma_{s}$-multialgebra $\mathcal{B}$}
\end{minipage}
\end{figure}

\subsection{\ldots being disconnected and having a strong basis}\label{disc. and basis}

We give an alternative approach to being $\textbf{cdf}$-generated: being disconnected and having a strong basis, in a sense we now define. 
%Our main motivation, when considering the notion of a strong basis, was to be able to weaken that condition. 
Remember that $\Sigma$-algebras with the universal mapping condition for the entire class of $\Sigma$-algebras (i.e. algebras of terms) are easier to be defined than the ones with the universal mapping property for some proper variety. That is why we start with multialgebras of terms, adopting the terminology ``strong basis''. Nonetheless, we still wish to be able in the future to define what would be a $\Sigma$-multialgebra satisfying the universal mapping property for some proper class of $\Sigma$-multialgebras, whose classical counterparts, on many contexts as that of vector spaces, have basis defined as minimal, not minimum, generating sets.
% This justifies the terminology ``strong basis'' we have adopted for the general case.
%When we first considered the notion of a strong basis, our motivation was to eventually be able to consider multialgebras satisfying a weaker condition instead. 
%In fact, we would like to be able to one day define "relatively free multialgebras", where basis probably are minimal (instead of minimum) generating sets.

\begin{definition}
We say $B\subseteq A$ is a {\it strong basis} of the $\Sigma$-multialgebra $\mathcal{A}=(A, \{\sigma_{\mathcal{A}}\}_{\sigma\in\Sigma})$ if it is the minimum of the set $\mathcal{G}=\{S\subseteq A\ : \  \langle S\rangle=A\}$ ordered by inclusion.
\end{definition}

\begin{example}
The set of variables $\mathcal{V}$ is a strong basis of $\textbf{T}(\Sigma, \mathcal{V})$.
\end{example}

\begin{example}
The set of elements without predecessor of a directed forest of height $\omega$ is a strong basis of the forest, considered as a $\Sigma_{s}$-multialgebra.
\end{example}

\begin{lemma}\label{ground cap span is in set}
For every subset $S$ of the universe of a $\Sigma$-multialgebra $\mathcal{A}$, $G(\mathcal{A})\cap \langle S\rangle\subseteq S$.
\end{lemma}

\begin{proof}
Suppose $x\in G(\mathcal{A})\cap \langle S\rangle$: if $x\notin S$, we will show that $x$ cannot be in $\langle S\rangle$, which contradicts our assumption. Indeed, if $x\notin S$ then 
\begin{equation*}x\notin\langle S\rangle_{0}=S\cup \bigcup_{\sigma\in\Sigma_{0}}\sigma_{\mathcal{A}},\end{equation*}
since $x\notin S$, and $x\in G(\mathcal{A})$ implies that 
\begin{equation*}x\in A\setminus B(\mathcal{A}) \subseteq A\setminus \bigcup_{\sigma\in\Sigma_{0}}\sigma_{\mathcal{A}}.\end{equation*}

Now, for induction hypothesis, suppose that $x\notin\langle S\rangle_{m}$. Then,
\begin{equation*}x\notin \langle S\rangle_{m+1}=\langle S\rangle_{m}\cup\bigcup \big\{\sigma_{\mathcal{A}}(a_{1}, \ldots  , a_{n}) \ : \  n\in\mathbb{N}, \, \sigma\in\Sigma_{n}, \, a_{1}, \ldots  , a_{n}\in \langle S\rangle_{m} \big\}\end{equation*}
since $x\notin \langle S\rangle_{m}$, and $x\in G(\mathcal{A})$ implies that 
\begin{equation*}x\in A\setminus B(\mathcal{A}) \subseteq A\setminus \bigcup \big\{\sigma_{\mathcal{A}}(a_{1}, \ldots  , a_{n}) \ : \  n\in\mathbb{N}, \, \sigma\in\Sigma_{n}, \, a_{1}, \ldots  , a_{n}\in \langle S\rangle_{m} \big\}.\end{equation*}
\end{proof}

\begin{theorem}\label{ground is in basis}
If the $\Sigma$-multialgebra $\mathcal{A}$ has a strong basis $B$, $G(\mathcal{A})\subseteq B$.
\end{theorem}

\begin{proof}
By Lemma~\ref{ground cap span is in set}, $G(\mathcal{A})=G(\mathcal{A})\cap A=G(\mathcal{A})\cap\langle B\rangle\subseteq B$.
\end{proof}

\begin{definition}
If $B$ is a strong basis of a disconnected $\Sigma$-multialgebra $\mathcal{A}$, we define the $B$-{\it order} of an element $a\in A$ as the natural number 
\begin{equation*}o_{B}(a)=\min \big\{k\in\mathbb{N}\ : \  a\in\langle B\rangle_{k}\big\}.\end{equation*}
\end{definition}

This is a clear generalization of the order, or complexity, of a term. In fact, the order of a term in $T(\Sigma, \mathcal{V})$ is exactly its $\mathcal{V}$-order.

It is clear that, if $a\in \sigma_{\mathcal{A}}(a_{1}, \ldots  , a_{n})$ and $o_{B}(a)\geq 1$, then $o_{B}(a_{1}), \ldots  , o_{b}(a_{n})<o_{B}(a)$. In fact, suppose $m+1=o_{B}(a)$, implying that
\begin{equation*}a\in \langle B\rangle_{m+1}=\langle B\rangle_{m}\cup \bigcup \big\{\sigma_{\mathcal{A}}(a_{1}, \ldots  , a_{n}) \ : \  n\in\mathbb{N}, \, \sigma\in\Sigma_{n}, \, a_{1}, \ldots  , a_{n}\in \langle B\rangle_{m} \big\}.\end{equation*}
Since $m=\min\{k\in\mathbb{N}\ : \  a\in\langle B\rangle_{k}\}$, we have that $a\notin \langle B\rangle_{m}$ and therefore
\begin{equation*}a\in \bigcup \big\{\sigma_{\mathcal{A}}(a_{1}, \ldots  , a_{n}) \ : \  n\in\mathbb{N}, \, \sigma\in\Sigma_{n}, \, a_{1}, \ldots  , a_{n}\in \langle B\rangle_{m} \big\}.\end{equation*}
Finally, we obtain that there exist $p\in\mathbb{N}$, $\theta\in\Sigma_{p}$ and $b_{1}, \ldots  , b_{p}\in \langle B\rangle_{m}$ such that $a\in \theta_{\mathcal{A}}(b_{1}, \ldots  , b_{p})$. Since $a\in \sigma_{\mathcal{A}}(a_{1}, \ldots  , a_{n})$, this implies that  $\sigma_{\mathcal{A}}(a_{1}, \ldots  , a_{n})\cap\theta_{\mathcal{A}}(b_{1}, \ldots  , b_{p})\neq\emptyset$, and therefore $p=n$, $\theta=\sigma$ and $b_{1}=a_{1}$, \ldots  , $b_{p}=a_{n}$, so that $o_{B}(a_{1}), \ldots  , o_{B}(a_{n})\leq m$.

But what if $a\in \sigma_{\mathcal{A}}(a_{1}, \ldots  , a_{n})$, for $n>0$, and $o_{B}(a)=0$, implying $a\in B$? We claim this case cannot occur, for if it does, 
\begin{equation*}B^{*}=\big(B\cup\{a_{1}, \ldots  , a_{n}\}\big)\setminus\{a\}\end{equation*}
generates $A$, while clearly not containing $B$. We have that $a\in \langle B^{*}\rangle_{1}$, since $a_{1}, \ldots  , a_{n}\in B^{*}$ and $a\in\sigma_{\mathcal{A}}(a_{1}, \ldots  , a_{n})$, and given that $B\setminus\{a\}\subseteq B^{*}$, it follows that $B\subseteq \langle B^{*}\rangle_{1}$, and so $\langle B\rangle_{0}\subseteq\langle B^{*}\rangle_{1}$. 

It is then true that $\langle B\rangle_{m}\subseteq \langle B^{*}\rangle_{m+1}$, for every $m\in\mathbb{N}$. Indeed,  if this is true for $m$, let $b\in \langle B\rangle_{m+1}$, and then either $b\in \langle B\rangle_{m}$, so that $b\in \langle B^{*}\rangle_{m+1}\subseteq \langle B^{*}\rangle_{m+2}$, or there exist $\theta\in\Sigma_{p}$ and $b_{1}, \ldots  , b_{p}\in \langle B\rangle_{m}$ such that $b\in\theta_{\mathcal{A}}(b_{1}, \ldots  , b_{p})$. In this case, since $\langle B\rangle_{m}\subseteq\langle B^{*}\rangle_{m+1}$, we have that
\begin{equation*}b\in \theta_{\mathcal{A}}(b_{1}, \ldots  , b_{p})\subseteq \bigcup \big\{\sigma_{\mathcal{A}}(a_{1}, \ldots  , a_{n}) \ : \  n\in\mathbb{N}, \, \sigma\in\Sigma_{n}, \, a_{1}, \ldots  , a_{n}\in \langle B^{*}\rangle_{m+1} \big\} \subseteq\langle B^{*}\rangle_{m+2},\end{equation*}
so once again $b\in\langle B^{*}\rangle_{m+2}$. Since $\langle B\rangle=\bigcup_{m\in\mathbb{N}}\langle B\rangle_{m}$ equals $A$, we have that $\langle B^{*}\rangle$ also equals $A$, as we previously stated. This is absurd, since $B$ is the minimum of $\{S\subseteq A\ : \  \langle S\rangle=A\}$, ordered by inclusion, and $B\not\subseteq B^{*}$. The conclusion must be that if $a\in\sigma_{\mathcal{A}}(a_{1}, \ldots  , a_{n})$ for $n>0$, then $o_{B}(a_{1}), \ldots  , o_{B}(a_{n})<o_{B}(a)$, regardless of the value of $o_{B}(a)$.

\begin{lemma}\label{ground is basis if disc}
If $\mathcal{A}$ is disconnected and has a strong basis $B$, then $B=G(\mathcal{A})$ and so $\mathcal{A}$ is generated by its ground.
\end{lemma}

\begin{proof}
Suppose $a\in B\setminus G(\mathcal{A})$. Since $a$ is in the build of $\mathcal{A}$, there exist $\sigma\in\Sigma_{n}$ and elements $a_{1}, \ldots  , a_{n}\in A$ such that $a\in\sigma_{\mathcal{A}}(a_{1}, \ldots  , a_{n})$. If $n>0$, $o_{B}(a)>o_{B}(a_{1})\geq 0$, which contradicts the fact that $a\in B$ and therefore $o_{B}(a)=0$.

If $n=0$, it is clear that $B^{*}=B\setminus \{a\}$ is a generating set smaller than $B$: generating set because, if $a\in \sigma_{\mathcal{A}}$, $a\in \bigcup_{\sigma\in\Sigma_{0}}\sigma_{\mathcal{A}}$ and therefore $B\subseteq \langle B^{*}\rangle_{0}$, so that $\langle B\rangle_{m}\subseteq\langle B^{*}\rangle_{m+1}$. This is also a contradiction, since $B$ is a strong basis.
\end{proof}

\begin{theorem}\label{theorem 3}
$\mathcal{A}$ is generated by its ground and disconnected iff it has a strong basis and it is disconnected.
\end{theorem}

\begin{proof}
We already proved, in Lemma~\ref{ground is basis if disc}, that if $\mathcal{A}$ is disconnected and has a strong basis $B$, then it is generated by its ground and disconnected. Conversely, if $\mathcal{A}$ is disconnected and generated by its ground, first of all it is clearly disconnected.

Now, if $\langle G(\mathcal{A})\rangle=A$, one has that $G(\mathcal{A})\subseteq S$ for every $S\in \{S\subseteq A\ : \  \langle S\rangle=A\}$, by Lemma~\ref{ground cap span is in set}. Therefore, the ground is a strong basis.
\end{proof}

Once again, we ask ourselves: does being disconnected imply having a strong basis or vice-versa? We show that this is not the case by providing examples of a multialgebra that is disconnected but does not have a strong basis, and one of a multialgebra that has a strong basis but is not disconnected.

\begin{example}
Take the signature $\Sigma_{s}$ and the $\Sigma_{s}$-multialgebra $\mathcal{C}$ from Example~\ref{C}.

We know that $\mathcal{C}$ is disconnected, but we also state that it does not have a strong basis: in fact, we see that the set $\big\{S\subseteq \{-1,1\}\ : \  \langle S\rangle=\{-1,1\}\big\}$ is exactly $\big\{\{-1\}, \{1\}, \{-1,1\}\big\}$, and this set has no minimum.
\end{example}

\begin{example}
Take the $\Sigma_{s}$-multialgebra $\mathcal{B}$ from Example~\ref{B}.

As we saw before, $\mathcal{B}$ is not disconnected. However we state that it has a strong basis: $B=\{0\}$ generates $\mathcal{B}$ and, since $\{1\}$ does not generate the multialgebra, we find that $B$ is a minimum generating set.
\end{example}

From these two examples, one could hypothesize that for a multialgebra being generated by its ground is equivalent to having a strong basis. Clearly, being generated by its ground implies having a strong basis, that is, the ground. But as we show in the example below, having a strong basis does not imply being generated by its ground.

\begin{example}
Take the signature $\Sigma_{s}$ from Example~\ref{s}, and consider the $\Sigma_{s}$-multialgebra $\mathcal{M}=(\{-1, 0, 1\}, \{s_{\mathcal{M}}\})$ such that $s_{\mathcal{M}}(0)=\{0\}$, $s_{\mathcal{M}}(1)=\{1\}$ and $s_{\mathcal{M}}(-1)=\{1\}$ (that is, $s_{\mathcal{M}}(x)=\{{\sf abs}(x)\}$, where ${\sf abs}(x)$ denotes the absolute value of $x$).

We have that $G(\mathcal{M})=\{-1\}$ and that $\langle\{-1\}\rangle=\{-1, 1\}$, so that $\mathcal{M}$ is not generated by its ground. But we state that $\{-1, 0\}$ is a strong basis. First of all, it clearly generates $\mathcal{M}$. Furthermore, the generating sets of $\mathcal{M}$ are only $\{-1, 0\}$ and $\{-1, 0, 1\}$, so that $\{-1, 0\}$ is in fact the smallest generating set.

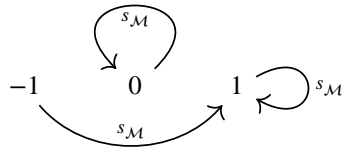
\begin{figure}[H]
\centering
\begin{tikzcd}
    -1 \arrow[rr, bend right=50]{}{s_{\mathcal{M}}}  & 0 \arrow[u, loop]{}{s_{\mathcal{M}}} & 1 \arrow[loop right, out=30, in=-30, distance=3em]{}{s_{\mathcal{M}}}
  \end{tikzcd}
\caption*{The $\Sigma_{s}$-multialgebra $\mathcal{M}$}
\end{figure}
\end{example}

\subsection{\ldots being disconnected and chainless}\label{disc. and chain.}

The last equivalence to being a submultialgebra of $\textbf{mT}(\Sigma, \mathcal{V}, \kappa)$ we give depends on the notion of being chainless, which is rather graph-theoretical in nature. Think of a tree that ramifies ever downward. One can pick any vertex and proceed, against the arrows, upwards until an element without predecessor is reached. More than that, it is not possible to find an infinite path, starting in any one vertex, by always going against the arrows: such a path, if it existed, would be what we shall call a chain. A multialgebra without chains is, very naturally, chainless.

As it was in the case of strong basis, there isn't a parallel concept to being chainless in universal algebra: it seems that this concept is far more natural when dealing with multioperations, although it can be easily applied to algebras if one wishes to do so. Closely related (although not equivalent) to chains are the branches in the formation trees of terms: if allowed to grow infinitely, these would became chains.

Given a permutation $\tau:\{1, \ldots  , n\}\rightarrow\{1, \ldots  , n\}$ in $S_{n}$, the group of permutations on $n$ elements, the action of $\tau$ in an $n$-tuple $(x_{1}, \ldots  , x_{n})\in X^{n}$ is given by
\begin{equation*}\tau(x_{1}, \ldots  , x_{n})=(x_{\tau(1)}, \ldots  , x_{\tau(n)}).\end{equation*}
Given $1\leq i, j\leq n$, we define $[i,j]$ to be the permutation such that $[i,j](i)=j$, $[i,j](j)=i$ and, for $k\in\{1, \ldots  , n\}$ different from $i$ and $j$, $[i,j](k)=k$.

\begin{definition}
Given a $\Sigma$-multialgebra $\mathcal{A}$, a sequence $\{a_{n}\}_{n\in\mathbb{N}}$ of elements of $A$ is said to be a {\it chain} if, for every $n\in\mathbb{N}$, there exist a positive natural number $m_{n}\in\mathbb{N}\setminus\{0\}$, a functional symbol $\sigma^{n}\in\Sigma_{m_{n}}$, a permutation $\tau_{n}\in S_{m_{n}}$ and elements $a_{1}^{n}, \ldots  , a_{m_{n}-1}^{n}\in A$ such that 
\begin{equation*}a_{n}\in \sigma^{n}_{\mathcal{A}}(\tau_{n}(a_{n+1}, a_{1}^{n}, \ldots  , a_{m_{n}-1}^{n})).\end{equation*}

A $\Sigma$-multialgebra is said to be {\it chainless} when it has no chains.
\end{definition}

\begin{example}
Take a directed forest of height $\omega$ and add a loop to it, that is, choose a vertex $v$ and add an arrow from $v$ to $v$. Then, $\{a_{n}\}_{n\in\mathbb{N}}$ such that $a_{n}=v$, for every $n\in\mathbb{N}$, is a chain.
\end{example}

\begin{example}
$\textbf{T}(\Sigma, \mathcal{V})$ is chainless.
\end{example}

\begin{lemma}\label{chainless implies gen by ground}
If $\mathcal{A}$ is chainless, then it is generated by its ground.
\end{lemma}

\begin{proof}
Suppose that $\mathcal{A}$ is not generated by its ground. Thus, $A\setminus\langle G(\mathcal{A})\rangle$ is not empty, and must therefore contain some element $a_{0}$. We create a chain $\{a_{n}\}_{n\in\mathbb{N}}$ by induction, the case $n=0$ being already done.

So, suppose we have created a finite sequence of elements $a_{0}, \ldots  , a_{k}\in A\setminus \langle G(\mathcal{A})\rangle$ such that, for each $0\leq n< k$, there exist a positive integer $m_{n}\in\mathbb{N}\setminus\{0\}$, a functional symbol $\sigma^{n}\in\Sigma_{m_{n}}$, a permutation $\tau_{n}\in S_{m_{n}}$ and elements $a_{1}^{n}, \ldots  , a_{m_{n}-1}^{n}\in A$ such that 
\begin{equation*}a_{n}\in \sigma^{n}_{\mathcal{A}}(\tau_{n}(a_{n+1}, a_{1}^{n}, \ldots  , a_{m_{n}-1}^{n})).\end{equation*}

Since $a_{k}\in A\setminus \langle G(\mathcal{A})\rangle$, we have that $a_{k}$ is not an element of the ground. So, there must exist $m_{k}\in\mathbb{N}$, a functional symbol $\sigma^{k}\in\Sigma_{m_{k}}$ and elements $b_{1}^{k}, \ldots  , b_{m_{k}}^{k}\in A$ such that
\begin{equation*}a_{k}\in\sigma^{k}_{\mathcal{A}}(b_{1}^{k}, \ldots  , b_{m_{k}}^{k}).\end{equation*}
Now, if all $b_{1}^{k}, \ldots  , b_{m_{k}}^{k}$ belonged to $\langle G(\mathcal{A})\rangle$, so would $a_{k}$: there must be an element $a_{k+1}\in\{b_{1}^{k}, \ldots  , b_{m_{k}}^{k}\}$, say $b_{l}^{k}$, such that $a_{k+1}\in A\setminus \langle G(\mathcal{A})\rangle$. We then define $a_{i}^{k}$ as $b_{j}^{k}$, for $j=\min\{i\leq p\leq m_{k}\ : \  p\neq l\}$ and $i\in\{1, \ldots  , m_{k}-1\}$, and 
\begin{equation*}\tau_{k}=[l-1, l]\circ\cdots\circ[1,2],\end{equation*}
and then it is clear that $\{a_{n}\}_{n\in\mathbb{N}}$ becomes a chain, with the extra condition that $\{a_{n}\}_{n\in\mathbb{N}}\subseteq A\setminus \langle G(\mathcal{A})\rangle$. Therefore $\mathcal{A}$ is not chainless.
\end{proof}

It becomes clear that a disconnected, chainless multialgebra is, by Lemma~\ref{chainless implies gen by ground}, disconnected and generated by its ground. We state, that, in fact, the converse also holds, when we arrive to yet another characterization of being a submultialgebra of $\textbf{mT}(\Sigma, \mathcal{V}, \kappa)$.

So, suppose $\mathcal{A}$ is disconnected and generated by its ground, and let $\{a_{n}\}_{n\in\mathbb{N}}$ be a chain in $\mathcal{A}$. Clearly no $a_{n}$ can belong to the ground, since 
\begin{equation*}a_{n}\in \sigma^{n}_{\mathcal{A}}(\tau_{n}(a_{n+1}, a_{1}^{n}, \ldots  , a_{m_{n}-1}^{n})),\end{equation*}
and therefore $o_{G(\mathcal{A})}(a_{n+1})<o_{G(\mathcal{A})}(a_{n})$, that is, the $G(\mathcal{A})$-order of $a_{n+1}$ is smaller than the $G(\mathcal{A})$-order of $a_{n}$. We obtain a contradiction, since if $o_{G(\mathcal{A})}(a_{0})=m$, then $o_{G(\mathcal{A})}(a_{m+1})<0$, what is impossible. Then, $\mathcal{A}$ must be chainless.

\begin{theorem}\label{theorem 4}
$\mathcal{A}$ is generated by its ground and disconnected iff it is chainless and disconnected.
\end{theorem}

Finally, Theorems \ref{sub mT is cdf-gen ground}, \ref{image of hom is sub}, \ref{cdf-gen iff gen by ground and disc}, \ref{theorem 3} and \ref{theorem 4} can be summarized as follows:

\begin{theorem} \label{charfreemulti}
Let $\mathcal{A}$ be a $\Sigma$-multialgebra. Are equivalent:

\begin{enumerate}
\item $\mathcal{A}$ is a submultialgebra of some $\textbf{mT}(\Sigma, \mathcal{V}, \kappa)$;

\item $\mathcal{A}$ is $\textbf{cdf}$-generated;

\item $\mathcal{A}$ is generated by its ground and disconnected;

\item $\mathcal{A}$ has a strong basis and is disconnected;

\item $\mathcal{A}$ is chainless and disconnected.
\end{enumerate}
\end{theorem}

This leads us to the second main notion introduced in the paper:

\begin{definition}
A {\em weakly free multialgebra} over $\Sigma$ is a multialgebra over $\Sigma$ satisfying any of the equivalent conditions of Theorem~\ref{charfreemulti}.
\end{definition}

% it can be seen that the proposed notion, when restricted to ordinary algebras (which corresponds to $\kappa=1$), recovers exactly the algebra of terms  $\textbf{T}(\Sigma, X)$.

By definition, weakly free multialgebras over $\Sigma$ coincide, up-to isomorphisms, with the submultialgebras of the members of the families $\mT(\Sigma, \mathcal{V})$, for some set $\mathcal{V}$ of generators  (recall Definition~\ref{free-mterms}).

It is important to stress the point that, although not all concepts present in the previous theorem have natural counterparts in universal algebra, by defining them for algebras presented as multialgebras we find that all of the conditions in the theorem are valid only for $\Sigma$-algebras of terms. This follows easily from the fact that the only $\textbf{cdf}$-generated algebras are the algebras of terms themselves. That is, weakly free algebras coincide with (absolutely) free algebras. 
Note that any subalgebra of $\textbf{T}(\Sigma, X)$ is of the form  $\textbf{T}(\Sigma, Y)$ for some $Y$. Thus, it can be observed that the generalization in $\textbf{MAlg}(\Sigma)$ of the collection of subalgebras of $\textbf{T}(\Sigma, \mathcal{V})$ corresponds to the class of submultialgebras of the members of the family $\mT(\Sigma, \mathcal{V})$. In turn, the meaning of $\textbf{T}(\Sigma, \mathcal{V})$ itself  is generalized in the category $\textbf{MAlg}(\Sigma)$ of $\Sigma$-multialgebras throughout the class of submultialgebras of the members $\mathcal{A}$ of $\mT(\Sigma, \mathcal{V})$ such that $G(\mathcal{A})=\mathcal{V}$.

Now, a few examples concerning being chainless, disconnected, having a strong basis and being generated by the ground will be given.

\begin{example}
Take the signature $\Sigma_{s}$ from Example~\ref{s}, and consider the $\Sigma_{s}-$multialgebra $\mathcal{Y}=(\mathbb{N}\cup\{a,b\}, \{s_{\mathcal{Y}}\})$ such that $s_{\mathcal{Y}}(n)=\{n+1\}$, for $n\in\mathbb{N}$, and $s_{\mathcal{Y}}(a)=s_{\mathcal{Y}}(b)=\{0\}$.

We see that $\mathcal{Y}$ is chainless since, given a chain $\{a_{n}\}_{n\in\mathbb{N}}$, it must be contained in the build of $\mathcal{Y}$, that is, $\mathbb{N}$: but then $a_{n+1}=a_{n}-1$, what is a contradiction, since there is only a finite number of elements smaller than $a_{0}$. At the same time, $\mathcal{Y}$ is not disconnected, since $s_{\mathcal{Y}}(a)=s_{\mathcal{Y}}(b)$.

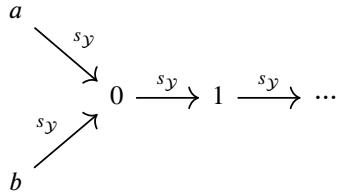
\begin{figure}[H]
\centering
\begin{tikzcd}
a\arrow[dr]{}{s_{\mathcal{Y}}} & & & \\
    & 0\arrow[r]{}{s_{\mathcal{Y}}} & 1\arrow[r]{}{s_{\mathcal{Y}}} & \cdots\\
b\arrow[ur]{}{s_{\mathcal{Y}}} & & & \\
  \end{tikzcd}
\caption*{The $\Sigma_{s}$-multialgebra $\mathcal{Y}$}
\end{figure}
\end{example}

\begin{example}
Take the $\Sigma_{s}$-multialgebra $\mathcal{C}$ from Example~\ref{C}.

We know that $\mathcal{C}$ is disconnected, however it is also not chainless: in fact, $\{(-1)^{n}\}_{n\in\mathbb{N}}$ and $\{(-1)^{n+1}\}_{n\in\mathbb{N}}$ are chains in $\mathcal{C}$.
\end{example}

As we saw, being chainless implies being generated by its ground and having a strong basis. The converse, however, is not true.

\begin{example}
Take the  $\Sigma_{s}$-multialgebra $\mathcal{B}$ from Example~\ref{B}.

We have already established that $\mathcal{B}$ has a strong basis and is generated by its ground, $\{0\}$, yet it is not chainless: $\{ 1\}_{n\in\mathbb{N}}$ is a chain in $\mathcal{B}$.
\end{example}

\section{Multialgebras cannot satisfy the universal mapping property}\label{UMP}

Now, we turn to a somewhat folkloric result: the category of multialgebras does not have free objects. This is equivalent to saying that there does not exist multialgebras satisfying the universal mapping property for the class of all $\Sigma$-multialgebras, or better yet, that the forgetful functor from this category to $\textbf{Set}$ does not have a left adjoint. Of course, such a result can be stated in various ways, depending on the adopted definition of homomorphism and even on the definition of multialgebra to be considered. So, we offer what we consider to be a simple proof of such result for the category $\textbf{MAlg}(\Sigma)$ as we have defined it.

\begin{definition}
A $\Sigma$-multialgebra $\mathcal{A}=(A, \{\sigma_{\mathcal{A}}\}_{\sigma\in\Sigma})$ satisfies the {\it universal mapping property} for the class of all $\Sigma$-multialgebras, over a set $X\subseteq A$ if, for every $\Sigma$-multialgebra $\mathcal{B}=(B, \{\sigma_{\mathcal{B}}\}_{\sigma\in\Sigma})$ and map $f:X\rightarrow B$, there exists a unique homomorphism $\overline{f}:\mathcal{A}\rightarrow\mathcal{B}$ extending $f$.
\end{definition}

In other words, if $j:X\rightarrow A$ is the inclusion, there exists only one homomorphism $\overline{f}:\mathcal{A} \to \mathcal{B}$ commuting the following diagram in {\bf Set}.
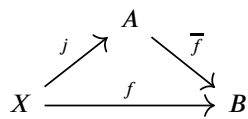
\begin{equation*}
\begin{tikzcd}
 & {A} \arrow{dr}{\overline{f}} \\
X \arrow{ur}{j} \arrow{rr}{f} && {B}
\end{tikzcd}
\end{equation*}

\begin{prop}
If $\mathcal{A}$ and $\mathcal{B}$ satisfy the universal mapping property for the class of all $\Sigma$-multialgebras over, respectively, $X$ and $Y$ such that $|X|=|Y|$, then $\mathcal{A}$ and $\mathcal{B}$ are isomorphic.
\end{prop}

\begin{proof}
Since $X$ and $Y$ are of the same cardinality, there exists bijective functions $f:X\rightarrow Y$ and $g:Y\rightarrow X$ inverses of each other. Take the extensions $\overline{f}:\mathcal{A}\rightarrow\mathcal{B}$ and $\overline{g}:\mathcal{B}\rightarrow\mathcal{A}$ and we have that $\overline{g}\circ\overline{f}$ is a homomorphism extending $g\circ f=id$, the identity on $X$.

Since the identical homomorphism $Id_{\mathcal{A}}:\mathcal{A}\rightarrow \mathcal{A}$ also extends $id$, we have that $Id_{\mathcal{A}}=\overline{g}\circ\overline{f}$. In a similar way we have that $Id_{\mathcal{B}}=\overline{f}\circ\overline{g}$, so that $\mathcal{A}$ and $\mathcal{B}$ are isomorphic.
\end{proof}

This way we can refer ourselves to the single $\Sigma$-multialgebra satisfying the universal mapping property over $X$, up to isomorphisms.

Remember that we have defined $\textbf{MAlg}(\Sigma)$ as the category whose objects are exactly all $\Sigma$-multialgebras and for which, given $\Sigma$-multialgebras $\mathcal{A}$ and $\mathcal{B}$, $Hom_{\textbf{MAlg}(\Sigma)}(\mathcal{A}, \mathcal{B})$ is the set of all homomorphisms from $\mathcal{A}$ to $\mathcal{B}$. We will denote by $\mathcal{U}:\textbf{MAlg}(\Sigma)\rightarrow \textbf{Set}$ the forgetful functor.

\begin{lemma}\label{adjoint of forget}
The functor $F:\textbf{Set}\rightarrow \textbf{MAlg}(\Sigma)$, associating a set $X$ with a $\Sigma$-multialgebra satisfying the universal mapping property over $X$, which we will denote $FX$, and a function $f:X\rightarrow Y$ with the only homomorphism $\overline{f}:FX\rightarrow FY$ extending $f$, is a left adjoint of $\mathcal{U}$.
\end{lemma}

\begin{proof}
For $X$ a set and $\mathcal{A}$ a $\Sigma$-multialgebra with universe $A$ we consider the functions, indexed by pairs consisting of a $\Sigma$-multialgebra $\mathcal{A}$ and a set $X$, 
\begin{equation*}\Phi_{\mathcal{A}, X}:Hom_{\textbf{Set}}(X,\mathcal{U}\mathcal{A})\rightarrow Hom_{\textbf{MAlg}(\Sigma)}(FX, \mathcal{A})\end{equation*}
taking a map $f:X\rightarrow A$ to the only homomorphism $\overline{f}:FX\rightarrow \mathcal{A}$ extending $f$. Each $\Phi_{\mathcal{A}, X}$ is clearly a bijection given that $FX$ satisfies the universal mapping property over $X$.

Now, given sets $X$ and $Y$, $\Sigma$-multialgebras $\mathcal{A}$ and $\mathcal{B}$, a function $f:Y\rightarrow X$ and a homomorphism $h:\mathcal{A}\rightarrow \mathcal{B}$, we have only to prove that the following diagram commutes in {\bf Set}.
\begin{equation*}
\begin{tikzcd}[sep=huge]
Hom_{\textbf{Set}}(X, \mathcal{U}\mathcal{A}) \arrow{r}{\Phi_{\mathcal{A}, X}} \arrow{d}{Hom(f, \mathcal{U}h)}& Hom_{\textbf{MAlg}(\Sigma)}(FX, \mathcal{A}) \arrow{d}{Hom(Ff, h)} \\%
Hom_{\textbf{Set}}(Y, \mathcal{U}\mathcal{B}) \arrow{r}{\Phi_{\mathcal{B}, Y}}& Hom_{\textbf{MAlg}(\Sigma)}(FY, \mathcal{B})
\end{tikzcd}
\end{equation*}

So we take a function $g:X\rightarrow \mathcal{U}\mathcal{A}$. Taking the upper right side of the diagram we have $\Phi_{\mathcal{A}, X}g=\overline{g}$ and $Hom(Ff, h)\overline{g}=h\circ \overline{g}\circ Ff$; on the lower left one, $Hom(f, \mathcal{U}h)g=\mathcal{U}h\circ g\circ f$ and $\Phi_{\mathcal{B}, Y}\mathcal{U}h\circ g\circ f=\overline{\mathcal{U}h\circ g\circ f}$.

Now, both $h\circ \overline{g}\circ Ff$ and $\overline{\mathcal{U}h\circ g\circ f}$ are homomorphisms from $FY$ to $\mathcal{B}$ extending $\mathcal{U}h\circ g\circ f:Y\rightarrow \mathcal{U}\mathcal{B}$. For the second one this is obvious, for the first we take an element $y\in Y$ and see that
\begin{equation*}h\circ \overline{g}\circ Ff(y)=h\circ\overline{g}\circ f(y)=h\circ g\circ f(y)=\mathcal{U}h\circ g\circ f(y)\end{equation*}
since, respectively: $Ff=\overline{f}$ (and $\overline{f}$ extends $f$); $\overline{g}$ extends $g$ (which is defined on $X\ni f(y)$); and $\mathcal{U}h=h$, (considered only as a function between sets).

Given that $FY$ satisfies the universal mapping property over $Y$, we have that $h\circ \overline{g}\circ Ff=\overline{\mathcal{U}\varphi\circ g\circ f}$ and the diagram in fact commutes.
\end{proof}

\begin{theorem}\label{no f-gen multi}
Given a non-empty signature $\Sigma$ and a set $X$, there does not exist a $\Sigma$-multialgebra which satisfies the universal mapping property over $X$.
\end{theorem}

\begin{proof}
Suppose that $\mathcal{A}=(A, \{\sigma_{\mathcal{A}}\}_{\sigma\in\Sigma})$ satisfies the universal mapping property over $X$ and let $\mathcal{V}$ be a set that properly contains $X$, meaning that $\mathcal{V}\neq \emptyset$ and therefore that $\textbf{T}(\Sigma, \mathcal{V})$ is well defined. Then, take the identity function $j: X\rightarrow T(\Sigma, \mathcal{V})$, such that $j(x)=x$ for every $x\in X$, and the homomorphism $\overline{j}:\mathcal{A}\rightarrow\textbf{T}(\Sigma, \mathcal{V})$ extending $j$.

Now, take the identity function $id:\mathcal{V}\rightarrow T(\Sigma^{2}, \mathcal{V})$ and the collections of choices $C$ and $D$ from $\textbf{T}(\Sigma, \mathcal{V})$ to $\textbf{mT}(\Sigma, \mathcal{V}, 2)$ such that, for $\sigma\in\Sigma_{n}$,
\begin{equation*}C\sigma_{\alpha_{1}, \ldots  , \alpha_{n}}^{\beta_{1}, \ldots  , \beta_{n}}(\sigma\alpha_{1} \ldots  \alpha_{n})=\sigma^{0}\beta_{1} \ldots \beta_{n}\end{equation*}
and
\begin{equation*}D\sigma_{\alpha_{1}, \ldots  , \alpha_{n}}^{\beta_{1}, \ldots  , \beta_{n}}(\sigma\alpha_{1} \ldots  \alpha_{n})=\sigma^{1}\beta_{1} \ldots \beta_{n},\end{equation*}
and consider the only homomorphisms $id_{C}, id_{D}:\textbf{T}(\Sigma, \mathcal{V})\rightarrow\textbf{mT}(\Sigma, \mathcal{V}, 2)$ extending, respectively, $id$ and $C$, and $id$ and $D$, which we know to exist given that $\textbf{T}(\Sigma, \mathcal{V})$ is \textbf{cdf}-generated by $\mathcal{V}$. Since $id_{C}\circ \overline{j}, id_{D}\circ\overline{j}:\mathcal{A}\rightarrow \textbf{mT}(\Sigma, \mathcal{V}, 2)$ both extend the function $j':X\rightarrow T(\Sigma^{2}, \mathcal{V})$ such that $j'(x)=x$ for every $x\in X$ (recalling that $\mathcal{V}$ properly contains $X$), we have $id_{C}\circ \overline{j}=id_{D}\circ\overline{j}$.

Now, if $\alpha\in T(\Sigma, \mathcal{V})\setminus \mathcal{V}$, we have that there exist $\sigma\in \Sigma_{n}$, for some $n\in\mathbb{N}$, and elements $\alpha_{1}, \ldots  , \alpha_{n}\in T(\Sigma, \mathcal{V})$ such that $\alpha=\sigma\alpha_{1} \ldots  \alpha_{n}$. In this case, \begin{equation*}id_{C}(\alpha)=\sigma^{0}id_{C}(\alpha_{1}) \ldots   id_{C}(\alpha_{n})\neq \sigma^{1}id_{D}(\alpha_{1}) \ldots  id_{D}(\alpha_{n})=id_{D}(\alpha),\end{equation*}
given that the leading functional symbols are distinct. From this, $id_{C}$ and $id_{D}$ are always different outside of $\mathcal{V}$. 

Since $id_{C}\circ \overline{j}=id_{D}\circ\overline{j}$, we must have that $\overline{j}(A)\subseteq \mathcal{V}$, and this is absurd since we are assuming $\Sigma$ non-empty. Indeed, if $\Sigma_{0}\neq\emptyset$, for a $\sigma\in\Sigma_{0}$ and $a\in\sigma_{\mathcal{A}}$ we have that $\overline{j}(a)=\sigma$ is in $T(\Sigma, \mathcal{V})$, but not in $\mathcal{V}$. If it is another $\Sigma_{n}$ which is not empty, given $a\in A$ (which exists, given that the universes of multialgebras are assumed to be non-empty) we have that, for $b\in \sigma_{\mathcal{A}}(a, \ldots  , a)$, it holds that $\overline{j}(b)=\sigma(\overline{j}(a), \ldots  , \overline{j}(a))$, which is not in $\mathcal{V}$.

We must conclude that there are no multialgebras with the universal mapping property.
\end{proof}

\begin{cor}\label{no free obj}
The category $\textbf{MAlg}(\Sigma)$ does not have an initial object.
\end{cor}

\begin{proof}
We state that, if $\mathcal{A}$ is an initial object, $\mathcal{A}$ has the universal mapping property over $\emptyset$. In fact, for every $\Sigma$-multialgebra $\mathcal{B}$ and map $f:\emptyset\rightarrow B$, there exists a single homomorphism $!_{\mathcal{B}}:\mathcal{A}\rightarrow \mathcal{B}$ extending $f=\emptyset$, that is, the only homomorphism between $\mathcal{A}$ and $\mathcal{B}$. But  multialgebras with the universal mapping property do not exist, by Theorem~\ref{no f-gen multi},. This concludes the proof.
\end{proof}

\begin{theorem}
The forgetful functor $\mathcal{U}:\textbf{MAlg}(\Sigma)\rightarrow\textbf{Set}$ does not have a left adjoint.
\end{theorem}

\begin{proof}
For suppose we have a left adjoint $F:\textbf{Set}\rightarrow\textbf{MAlg}(\Sigma)$ of $\mathcal{U}$, so that $F$ has a left adjoint and is therefore cocontinuous. Since $\emptyset$ is the initial object in $\textbf{Set}$, we have that $F\emptyset$ must be an initial object in $\textbf{MAlg}(\Sigma)$, which does not exist by Corollary~\ref{no free obj}.
\end{proof}

\section{Conclusions and Future Work} \label{conclusions}

The results obtained along the paper indicate that multialgebras of terms constitute a rich topic of study, and deserve to be further analyzed. Their connections to the theories of graphs and of partial orders seem clear, and suggest other properties of these objects, and possibly other characterizations. Multialgebras have been used in order to get satisfactory non-deterministic semantics for some non-classical logics, in particular paraconsistent logics (see, for instance, \cite[Chapter~6]{CC16}, \cite{CFG} and~\cite{CT20}). From the present study, we hope to obtain, with the aid of $\textbf{mT}(\Sigma, \mathcal{V}, \kappa)$ (now seen as the multialgebra of propositional formulas) and its submultialgebras, new interpretations of existing semantics for logic systems and new semantics altogether. Clearly,  decision problems concerning these multialgebras become relevant and need to be addressed.

Finally, in what is possibly the most important open question concerning multialgebras of terms, we refer back to something we have already mentioned in  this text. In universal algebra, a $\Sigma$-algebra $\mathcal{A}$ has the universal mapping property for a variety $\mathbb{V}$ of $\Sigma$-algebras over a subset $X$ of its universe when, for every $\mathcal{B}$ in $\mathbb{V}$ and every function $f:X\rightarrow B$, there exists a unique homomorphism $f:\mathcal{A}\rightarrow\mathcal{B}$ extending $f$. These are known as the {\em relatively} free algebras, which can be obtained in a variety from a quotient of $\textbf{T}(\Sigma, \mathcal{V})$. Some questions naturally arise: are there analogous of $\textbf{cdf}$-generated multialgebras with respect to classes of multialgebras? If so, are they obtained in some reasonable way from the multialgebras of terms?

\paragraph{Acknowledgements}

The first author acknowledges support from  the  National Council for Scientific and Technological Development (CNPq), Brazil
under research grant 306530/2019-8. The second author was supported by a doctoral scholarship from CAPES, Brazil. 
%We would also like to thank Hugo Mariano, Darllan Pinto, Peter Arndt, Ana Cl\'audia Golzio and Kaique Matias for making suggestions that greatly improved the clarity of this exposition.

\end{document}